\documentclass[11pt,twoside,a4paper]{article}

\usepackage{fullpage}
\usepackage{amsmath}
\usepackage{amsthm}
\usepackage{mathrsfs}
\usepackage{amsfonts} 
\usepackage{graphicx}
\usepackage{hyperref}
\usepackage{epsfig}

\usepackage{amsmath,amssymb,amsfonts,amsthm,amstext,verbatim}
\usepackage{enumerate,color,graphicx,stmaryrd}
\usepackage{psfrag}  
\usepackage{url}
\usepackage{mdwlist} 
\usepackage{mathtools}
\usepackage{wasysym}   
\usepackage{mathrsfs}
\usepackage{hyperref}
\usepackage{epsfig}

\usepackage{floatrow}
\usepackage{subcaption}

\theoremstyle{plain}
\newtheorem{theorem}{Theorem}
\newtheorem{proposition}{Proposition}[section]
\newtheorem{lemma}[proposition]{Lemma}
\newtheorem{corollary}[proposition]{Corollary}
\theoremstyle{definition}
\newtheorem{definition}{Definition}[section]

\theoremstyle{remark}

\newcommand{\N}{\mathbb{N}}
\newcommand{\Z}{\mathbb{Z}}
\newcommand{\C}{\mathbb{C}}

\newcommand{\sle}{\mathrm{SLE}}

\renewcommand{\H}{\mathbb{H}}
\renewcommand{\S}{\mathrm{Strip}}
\newcommand{\Rect}{\mathrm{Rect}}

\newcommand{\ang}{\mathrm{Ang}}
\newcommand{\weight}{\mathrm{w}}
\newcommand{\wind}{\mathrm{wind}}

\newcommand{\SAW}{\mathrm{SAW}}

\newcommand{\ag}[1]{{\textcolor{red}{[AG: #1\marginpar{\textcolor{blue}{\bf $\bigstar$}}]}}}

\renewcommand\Re{\operatorname{Re}}

\newcommand{\Tri}{{\rm Tri}}

\newcommand{\sfH}{\mathsf H}

\title{Self-avoiding walk on~$\Z^2$ with Yang--Baxter weights: \\ universality of critical fugacity and 2-point function.}
\date{\today}

\author{Alexander Glazman \and Ioan Manolescu}

\begin{document}

\newcommand{\Addresses}{{
  \bigskip
  \footnotesize

  A.~Glazman, \textsc{Tel Aviv University, School of Mathematical Sciences, Tel Aviv, Israel.}\par\nopagebreak
 \texttt{glazman@tauex.tau.ac.il}

  \medskip

  I.~Manolescu, \textsc{Universit{\'e} de Fribourg, Fribourg, Switzerland}\par\nopagebreak
  \texttt{ioan.manolescu@unifr.ch}

}}


\maketitle

\begin{abstract}
We consider a self-avoiding walk model (SAW) on the faces of the square lattice~$\mathbb{Z}^2$. This walk can traverse the same face twice, but crosses any edge at most once. The weight of a walk is a product of local weights: each square visited by the walk yields a weight that depends on the way the walk passes through it. The local weights are parametrised by angles~$\theta\in[\frac{\pi}{3},\frac{2\pi}{3}]$ and satisfy the Yang--Baxter equation. The self-avoiding walk is embedded in the plane by replacing the square faces of the grid with rhombi with corresponding angles.


By means of the Yang-Baxter transformation, we show that the 2-point function of the walk in the half-plane does not depend on the rhombic tiling ({\em i.e.} on the angles chosen). In particular, this statistic coincides with that of the self-avoiding walk on the hexagonal lattice. Indeed, the latter can be obtained by choosing all angles $\theta$ equal to~$\frac{\pi}{3}$. 

For the hexagonal lattice, the critical fugacity of SAW was recently proved to be equal to~$1+\sqrt{2}$. We show that the same is true for any choice of angles. In doing so, we also give a new short proof to the fact that the partition function of self-avoiding bridges in a strip of the hexagonal lattice tends to~$0$ as the width of the strip tends to infinity. This proof also yields a quantitative bound on the convergence.
\end{abstract}

\section{Self-avoiding walk on~$\Z^2$ with Yang--Baxter weights}
\label{sec:intro}

In spite of the apparent simplicity of the model, few rigorous results are available for two dimensional self-avoiding walk. The main conjecture is the convergence of plane SAW to a conformally invariant scaling limit. The latter is shown~\cite{LSW} to be equal to~$\text{SLE} (8/3)$, provided the scaling limit exists and is conformally invariant. A natural way to attack this problem is via the so-called parafermionic observable (see below for a definition) and its partial discrete holomorphicity. H.~Duminil-Copin and S.~Smirnov~\cite{DCSmi10} used the parafermionic observable to prove that the connective constant for the hexagonal lattice is equal to~$\sqrt{2+\sqrt{2}}$, a result that had beed non-rigorously derived by B.~Nienhuis in~\cite{Nie82}.

Self-avoiding walk on the square lattice is not believed to be integrable, therefore it is not reasonable to expect any explicit formula for the connective constant in this case\footnote{The most recent numerical estimate for the connective constant of the square lattice was obtained in \cite{Gutt16} as $2.63815853032790(3)$; it does not allow to conclude whether the connective constant is an algebraic number.}, nor the existence of a well-behaved equivalent observable. However, one may study natural variations of the model, such as the weighted version presented here, that render it integrable. By integrability here we mean that the weights satisfy the Yang--Baxter equation. Similar integrable versions exist for all loop $O(n)$ models (see \cite{N90,CaIk,Gl}), we limit ourselves here to $n=0$, that is to self-avoiding walk.

These variations provide a framework to analyse the universality phenomenon, {\em i.e.} that the properties of the model at criticality do not depend on the underlying lattice. Though believed to generally occur, the universality of critical exponents on isoradial graphs was established only for the Ising model~\cite{CheSmi12}, percolation~\cite{GriMan14} and the random-cluster model~\cite{DumLiMan15}. The current paper is the first step towards universality of the self-avoiding walk. 

Here we address the natural question of comparison between the properties of regular self-avoiding walk on the hexagonal lattice and those of weighted self-avoiding walk on a more general rhombic tiling. We show that in the half-plane, the 2-point function between points on the boundary is the identical in the weighted and regular models. A main tool in our proof, as well as in~\cite{GriMan14,DumLiMan15}, is the Yang--Baxter transformation discussed in Section~\ref{sec:YangBaxter}. 

Let us now define the model.
Consider a series of angles~$\Theta= \{\theta_k\}_{k\in\N}$, where~$\theta_k\in [\pi/3,2\pi/3]$ for all~$k$. Denote by~$\H(\Theta)$ the right half-plane tiled with columns of rhombi of edge-length $1$ in such a way that all rhombi in the~$k$-th column from the left have upper-left angle~$\theta_k$. We regard $\H(\Theta)$ as a plane graph, and call edges the sides of each rhombus;
we will refer to such graphs as rhombic tilings. 
Embed $\H(\Theta)$ so that the origin $0$ is the mid-point of a vertical edge of the boundary. 
Denote by $\S_{T}(\Theta)$ the strip consisting of the~$T$ leftmost columns of~$\H(\Theta)$. 

A self-avoiding walk on~$\H(\Theta)$ is a simple curve $\gamma$ starting and ending at midpoints of edges, intersecting edges at right angles and traversing each rhombus in one of the ways depicted in Fig.~\ref{figWeights}. The weight~$\weight_\Theta(\gamma)$ of a self-avoiding walk~$\gamma$ is the product of weights associated to each rhombus; for a rhombus of angle~$\theta$ the weight, depending on the configuration of arcs inside it, takes one of the six possible values: 1, $u_1(\theta)$, $u_2(\theta)$, $v(\theta)$, $w_1(\theta)$, $w_2(\theta)$ (see Fig.~\ref{figWeights} for the correspondence between the local pictures and the weights). These are explicit functions of $\theta$, given below. When it is clear which angles are considered, we will usually omit the subscript~$\Theta$ and write~$\weight(\gamma)$.




\begin{figure}
    \centering
    \begin{subfigure}{.7\textwidth}
      \centering
      \includegraphics[scale=0.85]{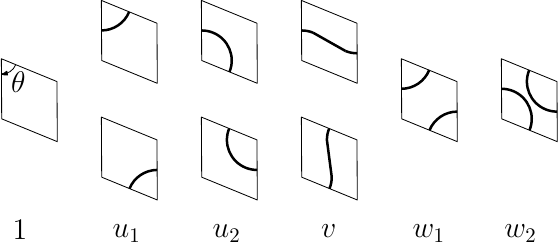}
    \end{subfigure}%
    \begin{subfigure}{.3\textwidth}
      \centering
      \includegraphics[scale=0.85,page=1]{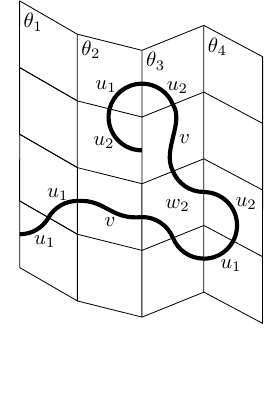}
    \end{subfigure}
    \caption{Different ways of passing a rhombus with their weights and an
        example of a walk of weight
        $u_1(\theta_1)^2 v(\theta_2)u_1(\theta_2)u_2(\theta_2)w_2(\theta_3)v(\theta_3)u_2(\theta_3)u_1(\theta_4)u_2(\theta_4)$ and length 
         $\frac{3}{\pi}\big[2\theta_1 + 3(\pi - \theta_3) + 7\big]$.}
    \label{figWeights}
\end{figure}

In 1990, Nienhuis~\cite{N90} computed the set of weights that are coherent with the Yang-Baxter equation for this model (see Section~\ref{sec:YangBaxter} for details).
These are:
\begin{align}\nonumber
    u_1&=\tfrac{\sin(\tfrac{5\pi}{4})\sin(\tfrac{5\pi}{8}+\tfrac{3\theta}{8})}{\sin(\tfrac{5\pi}{4}+\tfrac{3\theta}{8})\sin(\tfrac{5\pi}{8}-\tfrac{3\theta}{8})},&  
    u_2&=\tfrac{\sin(\tfrac{5\pi}{4})\sin(\tfrac{3\theta}{8})}{\sin(\tfrac{5\pi}{4}+\tfrac{3\theta}{8})\sin(\tfrac{5\pi}{8}-\tfrac{3\theta}{8})},&  
    v=\tfrac{\sin(\tfrac{5\pi}{8}+\tfrac{3\theta}{8})\sin(-\tfrac{3\theta}{8})}{\sin(\tfrac{5\pi}{4}+\tfrac{3\theta}{8})\sin(\tfrac{5\pi}{8}-\tfrac{3\theta}{8})}\\
    w_1&=\tfrac{\sin(\tfrac{5\pi}{8}+\tfrac{3\theta}{8})\sin(\tfrac{5\pi}{4}-\tfrac{3\theta}{8})}{\sin(\tfrac{5\pi}{4}+\tfrac{3\theta}{8})\sin(\tfrac{5\pi}{8}-\tfrac{3\theta}{8})},&
    w_2&=\tfrac{\sin(\tfrac{15\pi}{8}+\tfrac{3\theta}{8})\sin(-\tfrac{3\theta}{8})}{\sin(\tfrac{5\pi}{4}+\tfrac{3\theta}{8})\sin(\tfrac{5\pi}{8}-\tfrac{3\theta}{8})}.&
    \label{eq_w2c}
\end{align}
Notice that the weights above are all non-negative if and only if $\theta \in [\pi/3,2\pi/3]$. 
To have a probabilistic interpretation of the model, we limit ourselves to angles in this range. 
One may more generally define the model on any rhombic tiling, but certain walks may have negative weights (namely $w_1$ and $w_2$ are negative when $\theta > 2 \pi / 3$ and $\theta < \pi/3$, respectively). 

Henceforth, we always consider the weights listed above; the associated model will be referred to as the weighted self-avoiding walk. 
Replacing $\theta$  by $\pi-\theta$, effectively exchanges $u_1$ with $u_2$ and $w_1$ with $w_2$, but does not affect $v$. Hence, there is no ambiguity about which angles parametrise the rhombi. 

As explained in~\cite{Gl}, if $\theta = \pi/3$, then $w_2 = 0$ and $v = w_1 = u_2 = u_1^2$. Thus, any rhombus may be partitioned into two equilateral triangles, whose intersections with any walk is either void or one arc (see Fig. \ref{fig:split}). The weight generated by each rhombus may be computed as the product of two weights associated to the two triangles forming the rhombus, each contributing $1/{\sqrt{2 + \sqrt 2}}$ if traversed by an arc and $1$ otherwise. Thus, if $\Theta$ is the constant sequence equal to $\pi/3$, then each rhombus of $H(\Theta)$ may be partitioned into triangles, and $H(\Theta)$ becomes a triangular lattice (see Fig. \ref{fig:split}). The self-avoiding walk model described above becomes  that on the hexagonal lattice dual to the triangular one, with weight $1/(\sqrt{2 + \sqrt 2})^{|\gamma|}$ for any SAW $\gamma$ ($|\gamma|$ is the number of edges of $\gamma$). We call this the regular SAW, as it is the most common one. 

\begin{figure}
    \begin{center}
      \includegraphics[scale=0.8]{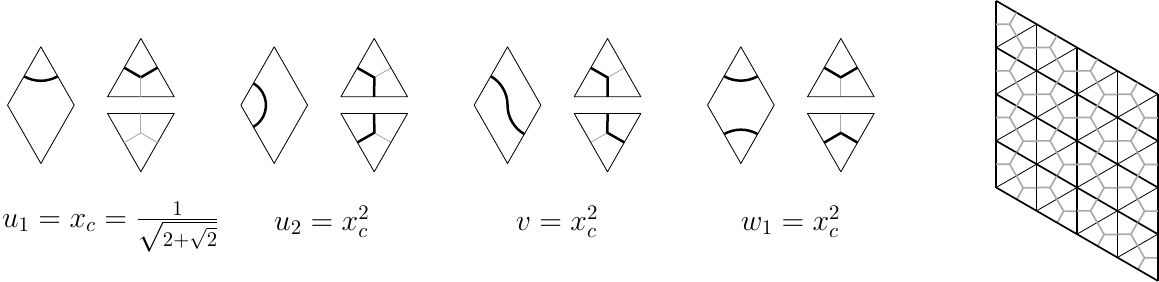}
    \end{center}
    \caption{A rhombus of angle $\pi/3$ is split into two equilateral triangles. 
    Any triangle contains at most one arc,  in which case it contributes $1/{\sqrt{2 + \sqrt 2}}$ to the weight.
    If all angles are equal to $\pi/3$, all faces of the rhombic tiling (bold black) maybe split into equilateral triangles, and walks may be viewed as regular self-avoiding walks on the hexagonal lattice (gray). 
}
    \label{fig:split}
\end{figure}

In~2009, Cardy and Ikhlef~\cite{CaIk} showed that for these weights, Smirnov's parafermionic observable (defined later in the text) is partially discretely holomorphic. Employing the original technics developed by Duminil-Copin and Smirnov~\cite{DCSmi10}, the first author generalised the calculation of the connective constant to the weighted self-avoiding walk \cite{Gl}. As a consequence, the weights \eqref{eq_w2c} may be considered critical for the weighted model. 


Given two points~$a$ and~$b$ with integer coordinates on the boundary of the right half-plane, the 2-point function between $a$ and $b$, denoted by~$G_\Theta(a,b)$, is the sum of weights of all walks from~$a$ to~$b$ on~$\H(\Theta)$ (see Fig.~\ref{fig:theta_tiling}):
\begin{align*}
G_\Theta(a,b) = \sum_{\gamma \text{ from $a$ to $b$}} \weight_\Theta(\gamma).
\end{align*}
By~$G_{\pi/3}(a,b)$ we denote the 2-point function when $\Theta$ is constant, equal to $\pi/3$. As mentioned above, this is the two point function of regular self-avoiding walk on a hexagonal lattice with edge-length~$1/\sqrt{3}$. 

\begin{theorem}\label{thm-saw-2-point}
Let~$\Theta= \{\theta_k\}_{k\in\N}$, where~$\theta_k\in [\pi/3,2\pi/3]$ for all~$k$. 
Then~$G_\Theta (a,b) = G_{\pi/3}(a,b)$ for any two points~$a$ and~$b$ on the boundary of the right half-plane.
\end{theorem}

\begin{figure}
    \begin{center}
    \includegraphics[scale=0.5,page=1]{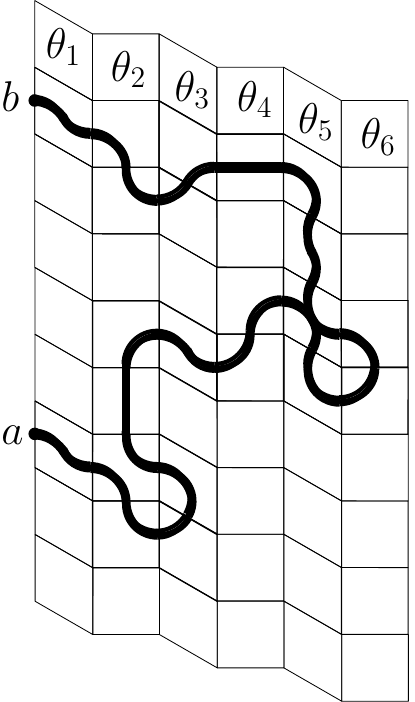}
    \hspace{40mm}
    \includegraphics[scale=0.5,page=2]{Theta_tiling2.pdf}
    \caption{
    {\it Left:} a path contributing to the 2-point funtion~$G_{\Theta}(a,b)$.
    {\it Right:} a bridge contributing to~$B_{6}(\Theta)$. }
    \label{fig:theta_tiling}
    \end{center}
\end{figure}

A bridge of width~$T$ is a SAW on $\S_{T}(\Theta)$, starting at $0$ and  ending on the right boundary of~$\S_{T}(\Theta)$ (see Fig.~\ref{fig:theta_tiling}). The partition function of bridges of width~$T$ is
\begin{align}\label{eq-def-B-theta-t}
	B_{T,\Theta} = \sum_{\gamma:\text{ bridge in } \S_{T}(\Theta)}{\weight_\Theta(\gamma)}\,,
\end{align}
where the sum is taken over all bridges of width $T$.

For the SAW on the hexagonal lattice it was shown that the total weight of bridges in a strip tends to~0 as the width of the strip tends to infinity \cite[Thm.~10]{BBDDG}. We give a new, short proof of this statement which also yields a quantitative bound on the convergence. 
\begin{proposition}\label{prop:saw-bridges-hex}
    We have
    \begin{align}\label{eq:B_sum}
    	\sum_{T \geq 1}\frac1T \big(B_{T,\tfrac{\pi}{3}}\big)^3  < \infty.
    \end{align}
    As a consequence, 
    the partition function of  self-avoiding bridges on the hexagonal lattice vanishes at infinity:
    $B_{T,\tfrac{\pi}{3}} \xrightarrow[T \to \infty]{} 0$. Moreover $B_{T,\tfrac{\pi}{3}} < 1/(\log T)^{1/3}$ for infinitely many values of~$T$.
\end{proposition}
The conjectural decrease of $B_T$ is much quicker than that implied by the above.  
Indeed it is expected that $B_T \sim T^{-1/4}$ as $T \to \infty$. For up to date numerical estimates on the asymptotics of $B_T$ see \cite[eq. (12)]{Gutt16}.

It is worth mentioning that the proof of Proposition~\ref{prop:saw-bridges-hex} uses certain symmetries of the hexagonal lattice (most notably the invariance under rotation by $\pi/3$). Hence this proof may not be applied directly to general rhombic tilings $H(\Theta)$. Nevertheless, using Theorem~\ref{thm-saw-2-point}, the part about convergence of~$B_T$ to zero can be extended to weighted self-avoiding walk on any rhombic tiling.

\begin{theorem}\label{thm-saw-bridges}
	Let~$\Theta= \{\theta_k\}_{k\in\N}$, where~$\theta_k\in [\pi/3,2\pi/3]$ for all~$k$. Then~$B_{T,\Theta} \xrightarrow[T\to \infty]{} 0$.
\end{theorem}

Our third result refers to self-avoiding walk with fugacity. 
Weighted self-avoiding walk with surface fugacity may be defined as was done in~\cite{BBDDG} for the regular model. 
In the half-plane, fugacity rewards (or penalises) walks whenever they approach the boundary by multiplying the weight by some~$y\ge 0$.  
Depending on the value of~$y$, a walk chosen with probability proportional to its weight will be either attracted to the boundary or repelled from it. The critical fugacity is the minimal $y$ such that self-avoiding walk with fugacity $y$ ``sticks'' to the boundary. 
This description is only illustrative, in fact the total weight of all self-avoiding walks in $\H(\Theta)$ is infinite \cite[Lemma 4.4]{Gl}, and no probability proportional to the weight exists.
A precise meaning of critical fugacity will be given below. 

In order to formally define critical fugacity, we deform the weight of a walk according to its length and its number of visits to the boundary.  
Let $\Theta = (\theta_k)_{k \geq 1}$ be a family of angles in $[\pi/3,2\pi/3]$ with $\theta_1 = \pi/3$. 
For a self-avoiding walk $\gamma$ on $\H(\Theta)$ define its length $|\gamma|$ as the sum of lengths of each arc, 
where the lengths of an arc spanning an angle $\theta$ is $\theta \frac{3}{\pi}$ and the length of any straight segment traversing a rhombus is 2. Notice that this definition is such that, when $\Theta$ is constant equal to $\pi/3$, the length of a walk is the number of edges in its representation on the hexagonal lattice. 
Further write $b(\gamma)$ for the number of times~$\gamma$ visits the leftmost column of rhombi as in Fig.~\ref{fig:CR}. More precisely, recall that each rhombus of the first column may be split into two equilateral triangles, each contributing to $\weight(\gamma)$ separately. Then $b(\gamma)$ is the number of visits of $\gamma$ to triangles adjacent to the boundary.


Given~$x,y\ge 0$, the $x$-deformed weight of a self-avoiding walk~$\gamma$ in~$\H(\Theta)$ with fugacity~$y$ is defined as
\begin{equation}\label{eq:weight-modified}
	\weight_\Theta (\gamma;x,y) = \weight_\Theta (\gamma)\cdot x^{|\gamma|} y^{b(\gamma)}.
\end{equation}
%

For~$x,y\geq 0$, the partition function of walks in~$\H(\Theta)$ with fugacity~$y$ is defined by:
\[
\SAW_\Theta (x,y)= \sum_{\gamma \text{ starts at } 0}{\weight_\Theta(\gamma; x,y)}\,.
\]
\begin{definition}
The critical fugacity~$y_c(\Theta)$ is the positive real number defined by
\begin{align}\label{eq:yc_def}
		y_c(\Theta) = \sup \big\{y \geq 0 \, | \, \forall 0<x<1, \SAW_\Theta (x,y) < \infty \big\}\,.
\end{align}
\end{definition}

In~\cite{BBDDG} it was proven that the critical fugacity for the regular self-avoiding on the hexagonal lattice is equal to~$1+\sqrt{2}$. We prove that the same is true for the self-avoiding walk with integrable weights, given that the rhombi in the first column are of angle~$\pi/3$. 

\begin{theorem}\label{thm-fugacity}
	Let~$\Theta= \{\theta_k\}_{k\in\N}$, where~$\theta_1= \pi/3$ and~$\theta_k\in [\pi/3,2\pi/3]$ for~$k>1$. 
	Then~$y_c (\Theta) = 1+\sqrt{2}$.
\end{theorem}



Let us briefly comment on the definition of the critical fugacity.
%
As already mentioned, the partition function of all walks with $x = y =1$ is infinite. Let~$x=1$ and~$y>$. Add one more column on the left and consider paths crossing only one rhomubs in it. The sum of their weigths is equal to~$v\cdot y$ times~$\SAW_\Theta (1,1)$, i.e. it is infinite.
This implies directly that $\SAW_\Theta (1, y) = \infty$ for all $y >0$.
For fixed $y > 0$, write 
$$x_c(y)  = \sup\big\{ x \geq 0 : \ \SAW_\Theta (x,y) < \infty \big\}.$$
This definition mimics that of the inverse connective constant for walks with fugacity. 

When $y = 1$, that is when no fugacity is added, we have $\SAW_\Theta (x,1) < \infty$ for all $x < 1$ (see \cite[proof of Thm. 1.1]{Gl}), 
which is to say $x_c(1) = 1$. 
The same is true for all $y < y_c(\Theta)$.
When $y > y_c(\Theta)$, it follows directly from the definition of the critical fugacity that $x_c(y) <1$. 

Thus, a fugacity is supercritical if it affects the value of the ``connective constant'' of the model. 
This is exactly the definition of critical fugacity used in \cite{BBDDG}; 
we have avoided it here because the connective constant for the weighted model does not appear naturally. 

One may also define the critical fugacity~$y_c(T, \Theta)$ for a strip~$\S_{T}(\Theta)$ in the same way, simply by replacing $\SAW_\Theta(x,y)$ in \eqref{eq:yc_def} with the partition function of weighted self-avoiding walks in $\S_T(\Theta)$:
\[
	\SAW_{T,\Theta}(x,y)  = \sum_{\gamma\,:\, \text{ walk in } \S_{T}(\Theta)}{\weight(\gamma;x,y)}.
\]

Define also the partition function of weighted bridges by
\[
	B_{T,\Theta}(x,y)  = \sum_{\gamma\,:\, \text{ bridge in } \S_{T}(\Theta)}{\weight(\gamma;x,y)}.
\]
We will consider the above for $x = 1$ as a series in $y$.

\begin{proposition}\label{prop:y_c_strip}
    Let~$\Theta= \{\theta_k\}_{k\in\N}$, where~$\theta_1= \pi/3$ and~$\theta_k\in [\pi/3,2\pi/3]$ for~$k>1$. 
    Then~$y_c(T,\Theta)$ is equal to the radius of convergence of~$B_{T,\Theta}(1;y)$, and
    $$y_c(T,\Theta)\xrightarrow[T\to \infty]{} y_c=1+\sqrt{2}.$$
\end{proposition}

\subsection*{Acknowledgements}
The authors would like to thank Hugo Duminil-Copin for pointing out that the strategy used in Proposition \ref{prop:saw-bridges-hex} yields an explicit bound on $B_T$. 

The first author was supported by Swiss NSF grant P2GE2\_165093, and  partially supported by the European Research Council starting grant 678520 (LocalOrder); he is grateful to the university of Tel-Aviv for hosting him. 
The second author is member of the NCCR SwissMAP. 
Part of this work was performed at IMPA (Rio de Janeiro) and at the University of Geneva, we are grateful to both institutions for their hospitality.

\section{Parafermionic observable}\label{sec:parafermion}

To analyse the behaviour of the self-avoiding walk  we will use the~\emph{para\-fermionic observable} introduced by Smirnov in~\cite{Smi_Ising} and its modification to incorporate fugacity introduced in~\cite{BBDDG}. The contour integral of this observable around each rhombus vanishes everywhere except for the part of the boundary where the surface fugacity is inserted. This leads to relations between the partition functions of arcs and bridges that are crucial for our proof.

\subsection{Observable without fugacity}

Fix a sequence $\Theta$ as before.  
The rows of rhombi of~$\H(\Theta)$ and~$\S_T(\Theta)$ may be numbered in increasing order by $\Z$, with the row $0$ containing the origin on its left boundary. 
Let~$\Rect_{T,L}(\Theta)$ be the rectangular-type domain consisting of the rows $-L,\dots, L$ of~$\S_T(\Theta)$ 
(see Fig.~\ref{fig:CR}).
Denote by $V(\Rect_{T,L}(\Theta))$ the set of all midpoints of the sides of the rhombi in $\Rect_{T,L}(\Theta)$ 
and by $V(\partial\Rect_{T,L}(\Theta))$ the points of~$V(\Rect_{T,L}(\Theta))$ lying on edges of~$\partial\Rect_{T,L}(\Theta)$ (that is edges of~$\Rect_{T,L}(\Theta)$ which are only adjacent to one rhombus of $\Rect_{T,L}(\Theta)$). Notice that the embedding ensures that $0 \in V(\partial \Rect_{T,L}(\Theta))$.

\begin{figure}
    \begin{center}
    \includegraphics[scale=0.75,page=1]{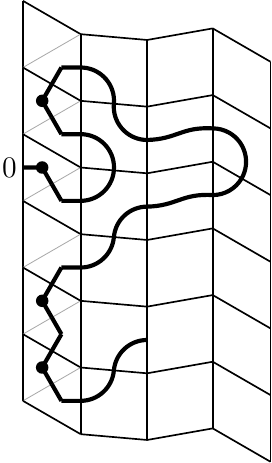}
    \hskip15mm
    \includegraphics[scale=0.8]{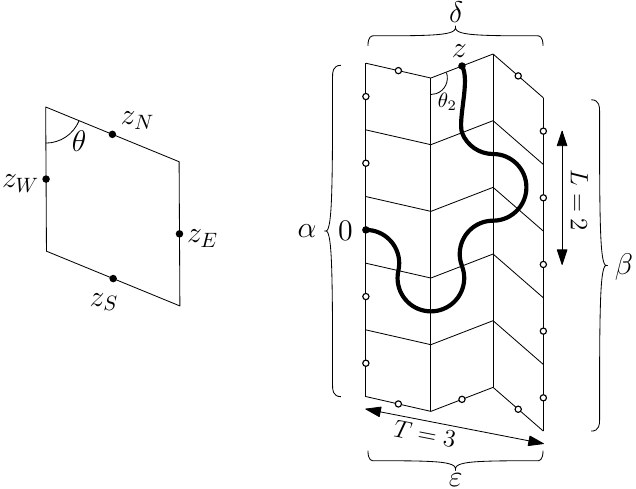}
    \caption{{\it Left:} When $\theta_1 =\pi/3$, the rhombi of the first column may be split into two equilateral triangles. Only visits to the triangles adjacent to the boundary (marked by dots) are counted in $b(\gamma)$. Here $b(\gamma) = 4$.
    {\it Middle:} A rhombus of angle~$\theta$ and mid-edges~$z_E,z_S,z_N,z_W$.
    {\it Right:} The domain $\Rect_{3,2}(\Theta)$ with the mid-points of boundary edges marked. 
    In bold: a path starting at $0$ and ending at a point in $z\in\delta$; its winding is $\theta_2$, as for any path ending at this point.}
    \label{fig:CR}
    \end{center}
\end{figure}

The parafermionic observable in the domain $\Rect_{T,L}(\Theta)$ (with no fugacity) is the function~$F$ defined on~$V(\Rect_{T,L}(\Theta))$ by
\begin{align}\label{def-parafermion}
	F(z)=\sum_{\gamma:0\to z}{\weight(\gamma)e^{-i\cdot \tfrac{5}{8}\cdot \wind(\gamma)}}\qquad \forall z\in V(\Rect_{T,L}(\Theta)),
\end{align}
where the sum runs over all self-avoiding walks $\gamma$ contained in $\Rect_{T,L}(\Theta)$, starting at~$0$ and ending at~$z$. Above, $\wind(\gamma)$ denotes the winding of~$\gamma$, i.e. the total angle of rotation of~$\gamma$ going from~$0$ to~$z$ (recall that a walk crosses the sides of rhombi at right angles). For instance, the arc from~$z_W$ to~$z_N$ in Fig.~\ref{fig:CR} has winding~$\theta$ and the arc from~$z_W$ to~$z_S$ has winding~$\theta - \pi$. 
Since $\Rect_{T,L}(\Theta)$ is a finite region, the sum in the definition of $F$ is finite, hence well-defined. 

The value~$5/8$ is chosen to render the contour integrals of $F$ null. 
It is specific to the self-avoiding walk model; 
similar observables exist for other models, where $5/8$ should be replaced with different values, see~\cite{Smi_Icm06,CaIk} and~\cite{DCSmi11} for a survey.

The {\em partial discrete holomorphicity} stated in the next lemma is a crucial property of the parafermionic observable. The word {\em partial} here refers to the fact that Eq.~\eqref{eq-CR} below can be viewed as the property that the contour integral around each rhombus vanishes, though the analogous property around each vertex does not hold.
The parafermionic observable was first introduced by Smirnov for the FK-Ising model~\cite{Smi_Ising}, where it satisfies stronger relations and in particular the contour integral around each face and around each vertex vanishes. In the FK-Ising model this observable used to prove the convergence of interfaces to~$\sle$ curves. Later, partial discrete holomorphicity was proved in case of the loop~$O(n)$ model~\cite{DCSmi10} on the hexagonal lattice and for the more general loop~$O(n)$ model with integrable weights~\cite{CaIk}. 
Here we state the partial discrete holomorphicity in the form given in~\cite[Lemma 3.1]{Gl}.

\begin{lemma}\label{lem-CR}
    The parafermionic observable~$F$ satisfies the following relation for each rhombus of~$\Rect_{T,L}(\Theta)$:
    \begin{align}\label{eq-CR}
   		F(z_E)-F(z_W)=e^{i\theta}(F(z_S)-F(z_N)),
    \end{align}
    where~$z_E$, $z_S$, $z_W$ and~$z_N$ are the midpoints of the edges of the rhombus, distributed as in Fig.~\ref{fig:CR}.
\end{lemma}

Equation~\eqref{eq-CR} is reminiscent of the Cauchy--Riemann equations for holomorphic functions;
it may also be written as the contour integral of $F$ around any rhombus being null. 
Summing the real part of~\eqref{eq-CR} over all rhombi in a particular domain yields a relation on the partition function analogous to that of~\cite{DCSmi10}[Lemma~2].  Denote the left, right, up and bottom boundaries of~$\Rect_{T,L}(\Theta)$ by~$\alpha$, $\beta$, $\delta$ and $\varepsilon$, respectively. We will use the following notation:
\begin{align}
	\label{eq_def_AB}
    A_{T,L,\Theta}&=\sum_{\gamma:0\to z\in \alpha}{\weight(\gamma)}\,,
    &B_{T,L,\Theta}&=\sum_{\gamma:0\to z\in \beta}{\weight(\gamma)}\,,\\
    \label{eq_def_DE}
    D_{T,L,\Theta}&=\sum_{\gamma:0\to z\in \delta}{\cos(\tfrac{3}{8}\wind(\gamma))\weight(\gamma)}\,,
    &E_{T,L,\Theta}&=\sum_{\gamma:0\to z\in \varepsilon}{\cos(\tfrac{3}{8}\wind(\gamma))\weight(\gamma)}\,.
\end{align}
The sums run over all self-avoiding walks in $\Rect_{T,L}(\Theta)$ ending at a point in $\alpha$, $\beta$, $\delta$ and $\varepsilon$, respectively.
The paths contributing to $A_{T,L,\Theta}$ are called (self-avoiding) arcs.


\begin{lemma}[Lem. 4.1 \cite{Gl}]\label{lem_eq_rectangle}
	For any sequence~$\Theta= \{\theta_k\}_{k\in\N}$ of angles between~$\tfrac{\pi}{3}$ and~$\tfrac{2\pi}{3}$,
    \begin{align}\label{eq_relation_rectangle}
    	\cos{\tfrac{3\pi}{8}}A_{T,L,\Theta} + B_{T,L,\Theta} + D_{T,L,\Theta} + E_{T,L,\Theta} = 1\,.
    \end{align}
\end{lemma}
The factor $1$ on the right-hand side of \eqref{eq_relation_rectangle} comes from the contribution to $F$ of the empty configuration, which is not accounted for in any of the terms on the left-hand side. In~\cite[Lem. 4.1]{Gl}, the case of a constant angle is considered. Here we are dealing with a general case and thus the factor~$\cos(\tfrac{3}{8}\wind(\gamma))$ appears in the definition of~$D_{T,L,\Theta}$ and~$E_{T,L,\Theta}$. However, the proof can be adapted mutatis mutandis and we do not give further details.

Write $A_{T,\Theta}$ and $ B_{T,\Theta}$ for the partition functions of arcs and bridges, respectively, in $\S_{T}(\Theta)$.

\begin{corollary}
	For any sequence~$\Theta= \{\theta_k\}_{k\in\N}$ of angles between~$\tfrac{\pi}{3}$ and~$\tfrac{2\pi}{3}$ and any $T \geq 1$, 
    \begin{align}\label{eq_relation_strip}
		\cos{\tfrac{3\pi}{8}}A_{T,\Theta} + B_{T,\Theta} = 1\,.
	\end{align}	
\end{corollary}

\begin{proof}
	Fix $\Theta$ and $T$ as in the statement. 
	First notice that $A_{T,\Theta} = \lim_{L \to \infty} A_{T,L,\Theta}$ and $B_{T,\Theta} = \lim_{L \to \infty} B_{T,L,\Theta}$. 
	Indeed, any self-avoiding arc of $\S_{T}(\Theta)$ is contained in a rectangle $\Rect_{T,L}(\Theta)$ for $L$ large enough, 
	and hence is accounted for in $A_{T,L,\Theta}$. Since all terms contributing to $A_{T,\Theta}$ are positive, the convergence is proved. 
	The same holds for bridges. 
	
	In light of \eqref{eq_relation_rectangle} and the above observation, it suffices to prove that 
	$D_{T,L,\Theta}\to 0$ and $E_{T,L,\Theta}\to 0$ as $L \to \infty$. 
	We will prove this for $D_{T,L,\Theta}$, the proof for $E_{T,L,\Theta}$ is identical.

	Observe that, any self-avoiding path $\gamma$ contributing to $D_{T,L,\Theta}$ may be completed by at most $T$ steps 
	(that is at most $T$ rhombi with arcs in them)  to form a self-avoiding path 
	on $\S_{T}(\Theta)$, with endpoints $(0,0)$ and $(0,L+1)$. Indeed, one can obtain such path by adding one more row of rhombi at the top of~$\Rect_{T,L}(\Theta)$ and linking the end of~$\gamma$ to the left side of~$\S_{T}(\Theta)$ by steps in this column.
	Each rhombus in the completion affects the weight of $\gamma$ by a factor bounded below by some universal constant $c > 0$. 
	Thus using that all angles~$\theta_k\in[\pi/3,2\pi/3]$ and hence~$\wind(\gamma)\geq \pi/3$ we get
	\begin{align*}
		0 \leq D_{T,L,\Theta} \leq c^T \cos\big(\tfrac{\pi}8\big)  G_{\S_T(\Theta)}(0,L+1).
	\end{align*}
	Finally observe that $$\sum_{L \in \Z} G_{\S_T(\Theta)}(0,L) = A_{T,\Theta} \leq\Big(\cos{\tfrac{3\pi}{8}}\Big)^{-1} < \infty,$$
	which implies that $G_{\S_T(\Theta)}(0,L+1)$ converges to $0$ as $L \to \infty$. 
	Since $T$ is fixed, the two displayed equation above imply that $D_{T,L,\Theta}\to 0$ as $L \to \infty$. 
\end{proof}

\subsection{Observable with fugacity}

\newcommand{\wt}{\widetilde}

The parafermionic observable with fugacity on the boundary was introduced in~\cite{BBDDG} for the hexagonal lattice. It may be adapted easily to our case; we do this below. 
The observable will be defined inside of rectangles and, for technical reasons, the fugacity will be inserted on the right boundary, rather than on the left. To mark this difference, we add a tilde to all quantities with fugacity on the right. 

Let~$\Theta = \{\theta_k\}_{k=1}^T$, with~$\theta_T= \pi/3$ and~$\theta_k\in [\pi/3, 2\pi/3]$ for~$1\leq k<T$. 
Consider~$\Rect_{T,L}(\Theta)$ and split the rhombi of the last column into equilateral triangles (see Fig.~\ref{fig:split}). 
For a SAW $\gamma$ on $\Rect_{T,L}(\Theta)$, define its weight as $\wt \weight(\gamma;1,y) = \weight(\gamma)y^{b_r(\gamma)}$, where~$b_r(\gamma)$ is the number of visits of~$\gamma$ to the triangles adjacent to the right boundary of~$\Rect_{T,L}(\Theta)$.
For $z \in V(\Rect_{T,L}(\Theta))$, set 
\begin{align}\label{def_parafermion_fug}
    \wt F(z; y)=\sum_{\gamma:0\to z}{\wt \weight(\gamma;1,y)e^{-i\cdot\tfrac{5}{8}\cdot  \wind(\gamma)}}.
\end{align}
It is  easy to check (following the same procedure as in \cite[Lemma~4.1]{Gl}) that this observable satisfies the same Cauchy-Riemann equation~\eqref{eq-CR} for all rhombi $r$ in columns $1,\dots, T-1$: 
\begin{align*}
   	&\wt F(z_E;y)-\wt F(z_W;y) - e^{i\theta}(\wt F(z_S;y)-\wt F(z_N;y)) = 0.
\end{align*}
However, for rhombi $r$ in the rightmost column, a ``defect'' needs to be added to the relation~\eqref{eq-CR}, which thus becomes
\begin{align*}
	&\Re\left[\wt F(z_E;y)-\wt F(z_W;y) - e^{i\theta}(\wt F(z_S;y)-\wt F(z_N;y))\right] 
	= &&\frac{(y-1)y^*}{y(y^* -1)}\cdot G_{\Rect_{T,L}(\Theta)}(0,z_E),
\end{align*}
where
\begin{align*}
	y^*  = 1 + \sqrt{2}\,.
\end{align*}

An analogous of Lemma \ref{lem_eq_rectangle} may be obtained by summing the real part of the equations above for all rhombi of $\Rect_{T,L}(\Theta)$. The result is analogous to~\cite[Prop.~4]{BBDDG}.
We first introduce notation analogous to~\eqref{eq_def_AB}--\eqref{eq_def_DE}; 
recall that the sides of $R_{T,L}(\Theta)$ are $\alpha,\beta,\delta$ and $\varepsilon$. Set
\begin{align*}
    \wt A_{T,L,\Theta}(y)&=\!\!\!\sum_{\gamma:0\to z\in \alpha}\!\!\!{\wt \weight(\gamma;1,y)}\,,
    &\wt B_{T,L,\Theta}(y)&=\!\!\!\sum_{\gamma:0\to z\in \beta}\!\!\!{\wt \weight(\gamma;1,y)}\,,\\
    \wt D_{T,L,\Theta}(y)&=\!\!\!\sum_{\gamma:0\to z\in \delta}\!\!\!{\cos\big(\tfrac{3}{8}\wind(\gamma)\big)\wt \weight(\gamma;1,y)}\,,
    &\wt E_{T,L,\Theta}(y)&=\!\!\!\sum_{\gamma:0\to z\in \varepsilon}\!\!\!{\cos\big(\tfrac{3}{8}\wind(\gamma)\big)\wt \weight(\gamma;1,y)}\,.
\end{align*}

\begin{lemma}\label{lem:rel_rectangle_fug}
    Let~$\Theta= \{\theta_k\}_{1 \leq k \leq T}$, where~$\theta_T= \pi/3$ and~$\theta_k\in [\pi/3,2\pi/3]$ for~$1 \leq k < T$. 
    Then, for any $y >0 $,
    \begin{align}\label{eq:relation_rectangle_fug}
	    \cos\big(\tfrac{3}{8}\big)\wt A_{T,L,\Theta}(1, y) + \tfrac{y^*-y}{y(y^*-1)}\wt  B_{T,L,\Theta}(1, y) + \wt D_{T,L,\Theta}(1, y) + \wt E_{T,L,\Theta}(1, y) = 1\,.
    \end{align}
   \end{lemma}

The proof of this lemma is similar to that of~\cite[Prop.~4]{BBDDG}; we will not detail it here. 
The only result of this section that will be used outside of it is the following corollary.

\begin{corollary}\label{cor:small_y}
    Let~$\Theta= \{\theta_k\}_{k\in\N}$, where $\theta_1 = \frac\pi3$ and~$\theta_k\in [\pi/3,2\pi/3]$ for all~$k \geq 2$. Assume that~$y< 1+\sqrt{2}$. 
    Then~$B_{T,\Theta}(y) \leq \tfrac{\sqrt{2}y}{1+\sqrt{2}-y}$.
\end{corollary}

\begin{proof}
	Fix a sequence $\Theta = (\theta_k)$ as above (with $\theta_1=\pi/3$), a value $T \geq 1$ and~$y <1+\sqrt{2}$. 
    Write $\wt\Theta = (\theta_T,\dots, \theta_1)$ and
    $\wt B_{T,\wt \Theta}(y)$ for the partition function of bridges in $\S_T(\wt \Theta)$ with fugacity $y$ on the right boundary:
    \begin{align*}
    	\wt B_{T,\wt \Theta}(y)&=\!\!\!\sum_{\gamma \text{ bridge in $\S_T(\wt\Theta)$}}\!\!\!{\wt \weight_{\wt \Theta}(\gamma;1,y)(\gamma)}.
    \end{align*}
	There is an obvious bijection between bridges in $\S_T(\wt\Theta)$ and those in $\S_T(\Theta)$: 
do a symmetry with respect to a vertical axis that exchanges the sides of the strip and shift it vertically so that it starts at row $0$. 
	The weight $\weight(\gamma)$ of any self-avoiding bridge $\gamma$ is equal to that of its reverse; moreover the winding of any bridge is $0$, whether it is in $\S_T(\wt\Theta)$ or $\S_T(\Theta)$. 
	Finally, if bridges in $\S_T(\wt\Theta)$ are weighted with fugacity $y$ on the right boundary, 
	that corresponds to bridges in $\S_T(\Theta)$ having fugacity on the left.
    Thus
    \begin{align*}
    	\wt B_{T,\wt \Theta}(y) = B_{T,\Theta}(y).
    \end{align*}
    Next we bound the left-hand side of the above.
    
    Fix some $L > 0$.
    All walks $\gamma$ in $\Rect_{T,L}(\wt\Theta)$ originating at $0$ and with endpoint on $\delta$ and $\varepsilon$ 
    have winding in $[\pi/3,2\pi/3]$ and $[-2\pi/3,-\pi/3]$, respectively. 
    Thus, all terms in~\eqref{eq:relation_rectangle_fug} are positive when~$y<y^*=1+\sqrt{2}$. We find 
    \[
    	\wt B_{T,L,\wt\Theta}(y) \leq \tfrac{y(y^*-1)}{y^*-y} = \tfrac{\sqrt{2}y}{1+\sqrt{2}-y} \,.
    \]
    Now observe that $\wt B_{T,\wt \Theta}(y) = \lim_{L \to \infty} \wt B_{T,L,\wt\Theta}(1, y)$. 
    Indeed, any bridge contributing to $\wt B_{T,\wt \Theta}(y)$ has a finite vertical span, and is therefore included in $\wt B_{T,L,\wt\Theta}(1, y)$ for $L$ large enough. Moreover, all terms in the sum defining $\wt B_{T,\wt \Theta}(y)$ are positive. 
     Since the bound for $\wt B_{T,L,\Theta}(y)$ above is uniform in $L$, it extends to $\wt B_{T,\wt \Theta}$ and thus to $B_{T,\Theta}$.
\end{proof}

\section{The Yang--Baxter equation}\label{sec:YangBaxter}

For this section only we will consider a slight generalisation of the model described above. 
First of all, we will consider rhombi with any angles in $(0,\pi)$. 
Secondly, we will consider walks on any rhombic tiling; 
rather than defining this properly, we direct the reader to the examples of figures \ref{fig_YB} and \ref{fig:saw-YB}.
Finally, we consider also families of walks rather than a single one. 
For $\gamma_1,\dots, \gamma_n$ a collection of (finite) self-avoiding walks 
such that all rhombi intersected by $\gamma_1 \cup \dots \cup \gamma_n$ are in one of the settings of Fig.~\ref{figWeights}, 
define the weight of the family as the product of the weights of each rhombus.

\begin{proposition}[Yang-Baxter equation]\label{prop:YB}
	Let $\sfH$ be a hexagon formed of three rhombi as in Fig.~\ref{fig_YB}, left diagram. 
	Write $\partial \sfH$ for the six boundary edges of $\sfH$. 
	Let $\sfH'$ be the rearrangement of the three rhombi that form $\sfH$, 
	as in Fig.~\ref{fig_YB}, middle diagram.
	For any~$k\leq 3$ and any choice of distinct vertices $x_1,y_1,\dots,x_k,y_k$,  on the edges of $\partial \sfH$, 
	\[
	\sum_{\substack{\gamma_1,\dots, \gamma_k \subset\sfH   \\ \gamma_i \colon x_i\to y_i}} \weight_{\sfH}(\gamma_1 \cup \dots \cup \gamma_k) 
	= \sum_{\substack{\gamma_1,\dots, \gamma_k \subset\sfH'\\ \gamma_i \colon x_i\to y_i}} \weight_{\sfH'}(\gamma_1 \cup \dots \cup \gamma_k) ,
	\]
	where the sum is taken over all disjoint paths~$\gamma_1,\dots,\gamma_k$.
	In other words, for any pairs of points on the boundary, the weight of walks connecting these pairs is the same in $\sfH$ and $\sfH'$. 
\end{proposition}
The proof consists simply of listing for each choice of $x_1,y_1,\dots,x_k,y_k$ ($k$ is always smaller than $3$)
the weights for all possible connections in the two tilings and explicitly computing their sum. 
The weights \eqref{eq_w2c} were derived in \cite{N90} to satisfy these equations.
Cardy and Ikhlef \cite{CaIk} found the same weights based on discrete holomorphicity.
The connection between the two was explored in \cite{AB}, where the Yang-Baxter equations are explicitly listed.

Equivalent relations may be obtained for any model with loop-weight between $0$ and $2$, with appropriate weight as functions of $n$ (see~\cite{Gl} for the exact formulae).
All three papers quoted above deal with general loop-weight; we only treat here the case of null loop-weight. 

As a consequence, if a large rhombic tiling contains three rhombi as in Fig.~\ref{fig_YB}, 
they may be rearranged without affecting the two point function for pairs of points outside of these three rhombi. 
 
\begin{corollary}
	Let~$\Omega$ be a rhombic tiling containing a hexagon $\sfH$ formed of three rhombi as in Proposition \ref{prop:YB}. 
	Denote by~$\Omega'$ the tiling that coincides with~$\Omega$ everywhere except for $\sfH$, 
	where the three rhombi are rearranged as $\sfH'$.
	Then, for any two vertices~$a,b$ of $\Omega$ that are not in $\sfH\setminus\partial\sfH$,	
	\begin{align}\label{eq:HH'}
		\sum_{\substack{\gamma\subset\Omega \\ \gamma \colon a\to b}} \weight_{\Omega}(\gamma) 
		= \sum_{\substack{\gamma\subset\Omega' \\ \gamma \colon a\to b}} \weight_{\Omega'}(\gamma).
	\end{align}
\end{corollary}

\begin{proof}
	We only sketch this. Write the sums in \eqref{eq:HH'} as double sums. 
	First sum over all possible configurations outside $H$ (and $H'$ respectively), 
	then over those inside $H$ (or $H'$) which lead to a single path connecting $a$ to $b$. 
	The inside sum on the right and left hand side is equal due to Proposition \ref{prop:YB};
	the outside weights are equal in $\Omega$ and $\Omega'$, since the two tilings are identical outside $H$ and $H'$, respectively.
\end{proof}

\begin{figure}[ht]
    \begin{center}
      \includegraphics[scale=0.9,page=1]{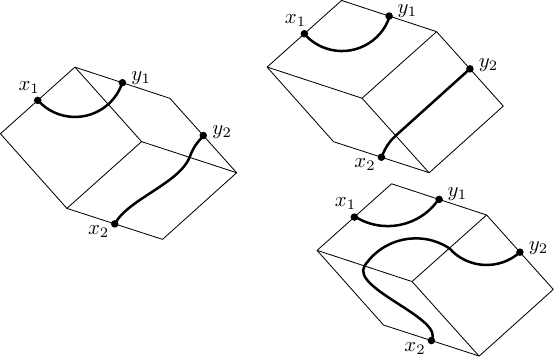}\qquad\quad
      \includegraphics[scale=0.55,page=1]{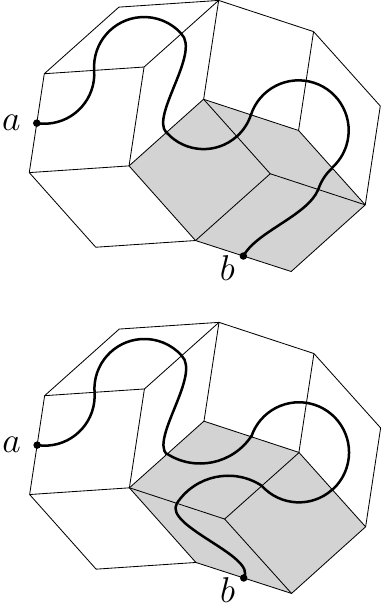}
    \end{center}
    \caption{{\em Left:} A hexagon formed of three rhombi of different angles. 
    {\em Middle:} The three rhombi may be rearranged to cover the same domain in a different fashion. 
    In the left image, the pairs of points $x_1,y_1$ and $x_2,y_2$ are connected in a single configuration; 
    in the middle image, the same connections are obtained in two distinct configurations. 
    The weight of the left configuration is equal to the sum of the weights of the two middle ones.
    {\em Right:} In a domain, changing three such rhombi does not alter the two point function between points $a$ and $b$.}
    \label{fig_YB}
\end{figure}

\section{Self-avoiding bridges and the 2-point function}\label{sec:bridges}

During the whole section we consider half-space rhombic tilings $H(\Theta)$.
Write $H(\pi/3)$ for the tiling with all angles equal to $\pi/3$. Recall that SAW on $H(\pi/3)$ is identical to that on the hexagonal lattice with the weight of a path $\gamma$ given by $(\sqrt{2+\sqrt{2}})^{-|\gamma|}$.

In this section we prove Theorems~\ref{thm-saw-2-point} and \ref{thm-saw-bridges}. 
Theorem~\ref{thm-saw-2-point} is shown by means of the Yang-Baxter transformation, which is used to gradually transform the lattice~$H(\pi/3)$ into an arbitrary lattice~$H(\Theta)$.  The relation \eqref{eq_relation_strip} between the partition functions of arcs and bridges in a strip together with Theorem~\ref{thm-saw-2-point} may be used to transfer the conclusion of Theorem~\ref{thm-saw-bridges} from the hexagonal lattice to any lattice $H(\Theta)$. Theorem~\ref{thm-saw-bridges} for the hexagonal lattice was proven in~\cite{BBDDG}; we provide below a new, shorter proof relying only on the parafermionic observable (see Proposition~\ref{prop:saw-bridges-hex}), that also provides an explicit (albeit weak) bound on $B_T$.

\subsection{Proof of Theorem \ref{thm-saw-bridges} for the hexagonal lattice.}

We will only work here with $H(\pi/3)$. Recall that weighted self-avoiding walk on $H(\pi/3)$ may be viewed as regular self-avoiding walk on a half space hexagonal lattice. We will write~$B_T$ instead of~$B_{T,\tfrac{\pi}{3}}$ for the partition function of bridges to simplify the notation.

Consider the strip $\S_{2L+1}(\pi/3)$ with width of~$2L+1$ hexagons and inscribe inside it an equilateral triangle $\Tri_{L}$ of side-length~$2L+1$ in such a way
that the midpoint of its vertical side is~0 (see Fig.~\ref{fig:saw-contour-triangle-strip}).
Let~$A^{\Delta}_{2L+1}$ be the partition function of walks starting at $0$, contained in the triangle, and ending on its left side;
write~$D^{\Delta}_{2L+1}$ for the partition function of walks ending on any of the two other sides of the triangle (see Fig.~\ref{fig:saw-contour-triangle-strip}).

\begin{lemma}
	The partition function $D^{\Delta}_{2L+1}$ is decreasing in $L$ and 
	\begin{align}\label{eq:triangle}
		B_{2L+1}  \leq \cos\left( \tfrac{\pi}{8}\right) D^{\Delta}_{2L+1} .
	\end{align}
\end{lemma}

\begin{proof}
    By summing the real part of \eqref{eq-CR} as in the proof of Lemma~\ref{lem_eq_rectangle}, we obtain:
    \[
    	\cos\left( \tfrac{3\pi}{8}\right)A^{\Delta}_{2L+1} + \cos\left( \tfrac{\pi}{8}\right) D^{\Delta}_{2L+1} = 1,
    \]
    where we used that the winding of all paths contributing to~$D^{\Delta}_{2L+1}$ is~$\pm\pi/3$.
    
    All walks contributing to $A^{\Delta}_{2L+1}$ also contribute to $A^{\Delta}_{2L+3}$, 
    which implies that $A^{\Delta}_{2L+1}$ is increasing in $L$. 
    By the above equation, $D^{\Delta}_{2L+1}$ is decreasing in $L$. 
    Moreover, $A^{\Delta}_{2L+1}  \leq A_{2L+1} $ since the latter partition function is over a larger set of walks. 
    By eq.~\eqref{eq_relation_strip}:
    \begin{align*}
   		B_{2L+1} = 1- \cos\left( \tfrac{3\pi}{8}\right) A_{2L+1} \leq 1-  \cos\left( \tfrac{3\pi}{8}\right) A^\Delta_{2L+1} =\cos\left( \tfrac{\pi}{8}\right) D^{\Delta}_{2L+1}.
    \end{align*}
    This provides the desired conclusion. 
\end{proof}

\begin{figure}[!ht]
    \begin{center}
    \includegraphics[width=.235\textwidth, page=1]{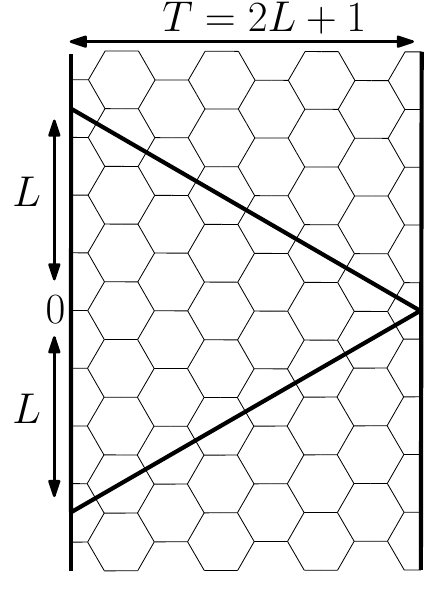}
    \hspace{0mm}
    \includegraphics[width=.235\textwidth, page=2]{hex_contour_triang_strip3.pdf}
    \hspace{0mm}
    \includegraphics[width=.235\textwidth, page=3]{hex_contour_triang_strip3.pdf}
    \hspace{0mm}
    \includegraphics[width=.235\textwidth, page=4]{hex_contour_triang_strip3.pdf}
    \end{center}
    \caption{
    {\it From left to right,} 
    {\it first:}~the strip of width~$T=2L+1$ and the equilateral triangle $\Tri_{L}$ of side-length~$2L+1$ inscribed in it.
    {\it Second:}~the same strip and three examples of walks: one arc contributing to~$A_T$ (blue) and two bridges contributing to~$B_T$.
    {\it Third:}~three examples of walks in $\Tri_{L}$:  one arc contributing to~$A^{\Delta}_T$ (blue) and two walks ending on the other sides 
    of the triangle and contributing to~$D^{\Delta}_T$. The one ending on the top contributes to $\ang_{L,K}^\Delta$ 
    {\it Fourth:}~the concatenation of (rotations and translations of) three walks contributing to  $\ang_{L,K_1}^\Delta$, $\ang_{K_1,K_2}^\Delta$ and $\ang_{K_2,K_3}^\Delta$, respectively, forms an arc contributing to $G_{\pi/3}(-L,K_3)$.}
    \label{fig:saw-contour-triangle-strip}
\end{figure}

We are in the position now to prove Proposition~\ref{prop:saw-bridges-hex}.

\begin{proof}[Proof of Proposition~\ref{prop:saw-bridges-hex}.]
By \eqref{eq:triangle} it suffices to show the conclusions of the proposition for  $D_{T}^\Delta$ instead of $B_T$. 

Recall the notation $G_{\pi/3}(a,b)$ for the 2-point function of walks on $H(\pi/3)$. 
By \eqref{eq_relation_strip}, $\lim_{T \to \infty} A_T \leq 1/\cos\left( \tfrac{3\pi}{8}\right)$. 
The limit above is the partition function of all arcs: 
$$ \lim_{T \to \infty} A_T = \sum_{k \in \Z} G_{\pi/3}(0,k) = 2\sum_{k \geq1} G_{\pi/3}(0,k).$$

For $L > 0$ and $0 \leq K \leq 2L$, write $\ang^\Delta_{L,K}$ for the partition function of walks in $\Tri_L$, 
starting at $0$ and ending on the top boundary, $K$ units from the left boundary (see fig \ref{fig:saw-contour-triangle-strip}). 
Then, by vertical symmetry, 
\begin{align}\label{eq:DDelta}
	D^\Delta_{2L+1} = 2 \sum_{K=0}^{2L}\ang^\Delta_{L,K}.
\end{align}

Fix $L > 0$. 
Using concatenations of walks contributing to $\ang^\Delta_{L,K}$ we may construct arcs contributing to 
$\sum_{b \geq  0} G_{\pi/3}(-L,b)$ as follows.
Divide the right half-plane $H(\pi/3)$ using the lines~$arg (z) =\pm \tfrac{\pi}{6}$ into three $\tfrac{\pi}{3}$-angles.
For $0 \leq K_3\leq 2K_2\leq 4 K_1\leq 8L$ and walks $\gamma^{(1)},\gamma^{(2)}, \gamma^{(3)}$ contributing to $\ang^\Delta_{L,K_1},\ang^\Delta_{K_1,K_2}$ and $\ang^\Delta_{K_2,K_3}$, respectively, obtain a walk contributing to $G_{\pi/3}(-L,K_3)$ by concatenating 
the translate of $\gamma_1$ by $(0,-L)$, 
the rotation by $\pi/3$ of the translate of $\gamma^{(2)}$ by $(0,K_1)$, and 
the rotation by $2\pi/3$ of the translate of $\gamma^{(3)}$ by $(0,K_2)$; see also fig~\ref{fig:saw-contour-triangle-strip}.
By summing over all values of $K_1,K_2,K_3$ we find
\begin{align*}
    \sum_{K_1=0}^{2L} \ang^\Delta_{L,K_1} \sum_{K_2=0}^{2K_1} \ang^\Delta_{K_1,K_2} \sum_{K_3 = 0}^{2K_2}  \ang^\Delta_{K_2,K_3} \leq  \sum_{k=L}^{9L}G_{\pi/3}(0,k).
\end{align*}
The sum on the right hand side goes from $L$ to $9L$ since the span of the obtained arc is $K_3 + L$, thus between $L$ and $9L$. 
Now, by \eqref{eq:DDelta}, the last sum on the left-hand side is equal to $\frac12 D^\Delta_{2K_2+1}$.
This is a decreasing quantity in $K_2$, thus 
\begin{align*}
    \sum_{k=L}^{9L}G_{\pi/3}(0,k)  
    &\geq \tfrac12\sum_{K_1=0}^{2L} \ang^\Delta_{L,K_1} \sum_{K_2=0}^{2K_1} \ang^\Delta_{K_1,K_2}\ D^\Delta_{2K_2 + 1} \\
    &\geq \tfrac12\sum_{K_1=0}^{2L} \ang^\Delta_{L,K_1}\ D^\Delta_{4K_1 +1} \sum_{K_2=0}^{2K_1} \ang^\Delta_{K_1,K_2}.
\end{align*}
By repeating this procedure for the other two sums, we find
\begin{align*}
    \sum_{k=L}^{9L}G_{\pi/3}(0,k)  
    &\geq \tfrac14\sum_{K_1=0}^{2L} \ang^\Delta_{L,K_1}\ D^\Delta_{4K_1 +1} \ D^\Delta_{2K_1 + 1}\\
    &\geq \tfrac14 \ D^\Delta_{8L+1}\  D^\Delta_{4L+1} \sum_{K_1=0}^{2L} \ang^\Delta_{L,K_1} 
    = \tfrac18 D^\Delta_{8L+1}\ D^\Delta_{4L+1}\ D^\Delta_{2L + 1}
    \geq  \tfrac18 \big(D^\Delta_{8L+1}\big)^3.
\end{align*}
Summing the above over $L= 9^k$ we find
\begin{align*}
	 \sum_{k=1}^{\infty}\big(D^\Delta_{8 \cdot 9^k +1 }\big)^3 \leq 8 \sum_{k=1}^{\infty}G_{\pi/3}(0,k) < \infty.
\end{align*}
Now, using the monotonicity in $T$ of $D^\Delta_T$ we may write 
\begin{align*}
	 \sum_{k=1}^{\infty}\big(D^\Delta_{8 \cdot 9^k  + 1}\big)^3  
	 \geq \sum_{k=1}^{\infty} \frac{1}{64 \cdot 9^k}\sum_{T = 8 \cdot 9^k + 1}^{8 \cdot 9^{k+1}}\big(D^\Delta_{T}\big)^3
	 \geq \frac{1}{8} \sum_{T = 73}^\infty \frac{1}{T}\big(D^\Delta_{T}\big)^3.
\end{align*}
Thus we have proved that $\frac{1}{T}\big(D^\Delta_{T}\big)^3$ is summable.
This implies in particular that 
$\big( D ^\Delta_{T}  (\log T)^{1/3}\big)_{T}$ contains a subsequence converging to $0$.
Finally, since $D^\Delta_{T}$ is decreasing, this implies $D^\Delta_{T} \xrightarrow[T \to \infty]{}0$.
The conclusions translate to $B_T$ using \eqref{eq:triangle}.
\end{proof}

\subsection{Proof of Theorem \ref{thm-saw-2-point} via the Yang-Baxter equation}

Now we are in the position to prove that the 2-point function is independent of the chosen tiling.
First we show that the 2-point function in a strip does not depend on the order of the columns of rhombuses in the tiling.
The strategy used here is reminiscent of the use of the Yang-Baxter equation to prove the commutation of transfer matrices, 
and of the strategy of \cite{GriMan14}.

\begin{figure}[!ht]
    \begin{center}
    \hspace{-5mm}
    \includegraphics[scale=0.5, page=1]{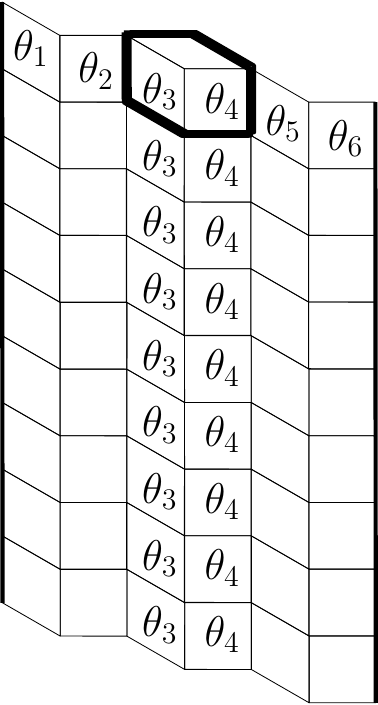}
    \hspace{5mm}
    \includegraphics[scale=0.5, page=2]{YB_transformation.pdf}
    \hspace{5mm}
    \includegraphics[scale=0.5, page=3]{YB_transformation.pdf}
    \hspace{5mm}
    \includegraphics[scale=0.5, page=4]{YB_transformation.pdf}
    \hspace{-5mm}
    \end{center}
    \caption{
    {\it Leftmost:} The domain $D_0$ obtained by adding a rhombus $r$ to the rectangle $\Rect_{T,L}(\Theta)$.
    {\it Second from the left:} The tiling $D_1$ is the result of the first Yang--Baxter transformation applied in the bold region of $D_0$.
    {\it Third from the left:} After two Yang--Baxter transformations $r$ is pushed down by~$2$ units and we obtain $D_2$. 
    {\it Rightmost:} After $2L$ repetitions, the rhombus $r$ is pushed all the way to the bottom of $\Rect_{T,L}(\Theta)$ and the two columns of rhombi are exchanged. The resulting tiling is $D_{2L}$.}
    \label{fig:saw-YB}
\end{figure}

\begin{proposition}\label{prop:saw-YB-strip}
    Let~$\S_T(\Theta)$ be a vertical strip tiled with~$T$ columns with angles~$\theta_1, \dots, \theta_T$.
    Then for any~$a,b$ on the boundary of~$\S_T(\Theta)$ the 2-point function~$G(a,b)$ does not depend on the order of angles.
\end{proposition}

The above applies both when $a,b$ are on the same side of~$\S_T(\Theta)$ as when they are on different sides. 
In the proof below, we make no particular assumption on the position of $a$ and $b$ other than that they are on the boundary. 

\begin{proof}
    Let $\S_T(\Theta)$ be a strip as in the statement of the proposition and $a,b$ be two points on its boundary.
    For $1 \leq i < T$ denote by $\tau_i$ the transposition of $i$ and $i+1$ and by $\Theta \circ \tau_i$ the sequence with $\theta_i$ and $\theta_{i+1}$ transposed:
    \begin{align*}
    	\Theta \circ \tau_i = (\theta_1, \dots,\theta_{i-1}, \theta_{i+1},\theta_i,\theta_{i+2},\dots \theta_T).
    \end{align*}
    In order to prove the proposition, it is sufficient to show that the partition function in $G(a,b)$ in $\S_T(\Theta)$ is equal to the one in $\S_T(\Theta\circ \tau_i)$. 

    This is done by means of the Yang--Baxter transformation, which transforms the rhombic tiling while preserving the partition function (see Section~\ref{sec:YangBaxter} and references therein for more details).
    
    Fix two points~$a$ and~$b$ on the boundary of~$\S_T(\Theta)$ and~$\varepsilon>0$.
    For the sake of this proof, if~$D$ denotes a simply connected subset of faces of $\S_T(\Theta)$ or $\S_T(\Theta\circ\tau_i)$ 
    that contains $a$ and $b$ on its boundary, 
    then write $G_{D}(a,b)$ for the two point function of walks in $D$ from $a$ to $b$: 
    \begin{align*}
    	G_D(a,b) =  \sum_{\substack{\gamma \text{ from $a$ to $b$};\\ \gamma \subset D}} \weight(\gamma).
    \end{align*}
    First observe that there exists $L >0$ such that 
    \begin{align*}
    	G_{\S_T(\Theta)}(a,b) - \varepsilon &\leq 	G_{\Rect_{T,L}(\Theta)}(a,b) \leq G_{\S_T(\Theta)}(a,b) \qquad \text{ and }\\
    	G_{\S_T(\Theta\circ \tau_i)}(a,b) - \varepsilon &\leq 	G_{\Rect_{T,L}(\Theta\circ \tau_i)}(a,b) \leq G_{\S_T(\Theta\circ \tau_i)}(a,b). 
    \end{align*}
    (Above we used that the 2-point function is finite, which is the case due to \eqref{eq_relation_strip}.) 
    Without loss of generality, we may suppose $\theta_i < \theta_{i+1}$ \footnote{If $\theta_i > \theta_{i+1}$, the rhombus may be added at the bottom and will be slid to the top using Yang--Baxter transformations. If $\theta_i = \theta_{i+1}$ the result is trivial. }.
    Let $D_0$ be the graph obtained by adding a rhombus~$r$ to $ \Rect_{T,L}(\Theta)$ at the top of the columns $i$ and $i+1$. 
    Precisely, the added rhombus has two sides equal to the top sides of the columns $i$ and $i+1$; 
    the condition $\theta_i < \theta_{i+1}$ ensures that $r$ does not overlap with the rhombi of $\Rect_{T,L}(\Theta)$, 
    and $D_0$ is a rhombic tiling (see Fig.~\ref{fig:saw-YB}). 
    Then we have 
    \begin{align*}
    	G_{D_0}(a,b) - G_{\Rect_{T,L}(\Theta)}(a,b)
		= \sum_{\substack{\gamma:a \to b \\ \gamma \text{ uses $r$}}} \weight(\gamma).
    \end{align*}
    A path $\gamma$ contributing to the above traverses $r$ only as one arc, hence always has positive weight. 
    In particular, $G_{D_0}(a,b) \geq G_{\Rect_{T,L}(\Theta)}(a,b)$.
    
    On the other hand, to any $\gamma$ as in the sum above, associate the walk $\gamma'$ in $\Rect_{T,L+1}(\Theta)$ that connects $a$ to $b$, 
    obtained by keeping the same configuration in $\Rect_{T,L}(\Theta)$ as in $D_0$ 
    and replacing the one arc in $r$ by two arcs in the top row of $\Rect_{T,L+1}(\Theta)$. 
    Then the ratio of the weight of $\gamma$ and $\gamma'$ is bounded above by some universal constant $c$.
    Thus 
    \begin{align*}
    	G_{D_0}(a,b) - G_{\Rect_{T,L}(\Theta)}(a,b)
		\leq c \big(G_{\Rect_{T,L+1}(\Theta)}(a,b) - G_{\Rect_{T,L}(\Theta)}(a,b)\big) <  c \cdot \varepsilon.
    \end{align*}
    
    Apply the Yang-Baxter transformation to the added rhombus and the two rhombi adjacent to it (notice that these indeed form a hexagon). 
    This in effect slides the added rhombus one unit down (see fig.~\ref{fig:saw-YB}). Call $D_1$ the resulting graph and conclude that 
    $$G_{D_0}(a,b) =     G_{D_1}(a,b).$$
	The operation may be repeated to slide the added rhombus one more unit downwards. 
	Performing $2L$ such Yang--Baxter transformations leads to 
	$$G_{D_0}(a,b) = G_{D_{2L}}(a,b),$$
	where $D_{2L}$ is the rhombic tiling $\Rect_{T,L}(\Theta\circ \tau_i)$ with the additional added rhombus at the bottom of columns $i$ and $i+1$. 
	
	By the same reasoning as above, 
	\begin{align*}
		0 \leq G_{D_{2L}}(a,b) - G_{\Rect_{T,L}(\Theta \circ \tau_i)}(a,b)
		\leq c \cdot \varepsilon.
	\end{align*}
	Thus, we conclude that 
	\begin{align*}
		c \cdot \varepsilon&> |G_{D_{2L}}(a,b) - 	G_{\Rect_{T,L}(\Theta \circ \tau_i)}(a,b) | \\
		&= |G_{D_{0}}(a,b) - 	G_{\Rect_{T,L}(\Theta \circ \tau_i)}(a,b) |\\
		&\geq 
		 | G_{\Rect_{T,L}(\Theta)}(a,b)  - G_{\Rect_{T,L}(\Theta \circ \tau_i)}(a,b) | - 
		  |G_{D_{0}}(a,b) - 	G_{\Rect_{T,L}(\Theta)}(a,b) |.
	\end{align*}
	The last term above is also bounded by $c\cdot\varepsilon$, and we find
	\begin{align*}
		 | G_{\S_{T}(\Theta)}(a,b) - G_{\S_{T}(\Theta \circ \tau_i)}(a,b) | 
		& \leq 
		 | G_{\S_{T}(\Theta)}(a,b)  - G_{\Rect_{T,L}(\Theta)}(a,b) | \\
		 &+| G_{\Rect_{T,L}(\Theta)}(a,b)  - G_{\Rect_{T,L}(\Theta \circ \tau_i)}(a,b) |\\
		 &+| G_{\S_{T}(\Theta\circ \tau_i)}(a,b)  - G_{\Rect_{T,L}(\Theta \circ \tau_i)}(a,b) | \leq  (2+2c)\varepsilon. 
	\end{align*}
	Since $\varepsilon$ may be chosen arbitrarily small, we find 
	$ G_{\S_{T}(\Theta)}(a,b) = G_{\S_{T}(\Theta \circ \tau_i)}(a,b) $, which is the desired conclusion. 
%
\end{proof}

Lemma~\ref{prop:saw-YB-strip} allows us to exchange columns of different angles but it does not permit to change the angles.
Next lemma deals with this question and tells us that the 2-point function in a strip decreases when one of the angles is replaced by~$\pi/3$.

\begin{lemma}\label{lem:saw-strip-monotonicity}
    Let $\Theta = (\theta_1,\dots, \theta_T)$ be a finite sequence of angles with $\theta_k \in [\pi/3,2\pi/3]$ for all $k$. 
	Then for any two points $a,b$ on the left boundary of $\S_{T}(\Theta)$ we have 
	\begin{align*}
		G_{\S_T(\Theta)} (a,b) \geq G_{\S_T(\theta_1,\theta_2,\dots,\theta_{T-1},\pi/3)} (a,b).
	\end{align*}
\end{lemma}

\begin{proof}
    Let $T,\Theta, a,b$ be as in the statement. 
    Write $\tilde\Theta $ for the sequence $(\theta_1,\theta_2,\dots,\theta_{T-1},\pi/3)$.
    We will show that any self-avoiding walk $\gamma$ from~$a$ to~$b$ in~$\S_T(\Theta)$ 
    has either the same or larger weight than its correspondent walk in $\S_T(\tilde \Theta)$.
    
    \begin{figure}
  	  	\centering
      	\includegraphics[scale=0.75]{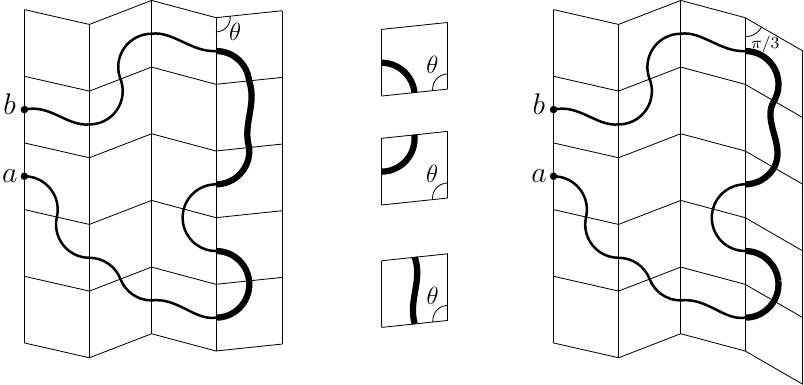}
    	\caption{An arc in $\S_T(\Theta)$ (left) and the corresponding arc in $\S_T(\tilde \Theta)$ (right). 
	The difference in weight comes from three types of rhombi depicted in the middle. The first two come in pairs and their combined weight is lowest when $\theta = \pi/3$; the third one has lowest weight when $\theta = \pi/3$.  }
    	\label{fig:YB_application}
	\end{figure}

    Indeed, consider any such walk~$\gamma$ in~$\S_T(\Theta)$. 
    The intersection of $\gamma$ with the rightmost column of~$\S_T(\Theta)$ 
    is formed of a family of disjoint arcs, as depicted in Fig. \ref{fig:YB_application}.
    Write $\chi_1,\dots, \chi_\ell$ for these arcs (take $\ell = 0$ if $\gamma$ does not visit column $T$). 
    The weight of each such arc only depends on $\theta_T$:
    an arc $\chi_j$ is formed of a rhombus of type $u_1$, a number $k \geq 0$ of rhombi or type $v$ 
    and one rhombus of type $u_2$; its weight is then 
    \begin{align}\label{eq:arc_weight}
    \weight_{\theta_T}(\chi_j) =  
    \frac{{\sin(\frac{5\pi}{4})\sin(\frac{5\pi}{8}+\frac{3\theta_T}{8})}
    \big[{\sin(\frac{5\pi}{8}+\frac{3\theta_T}{8})\sin(-\frac{3\theta_T}{8})}\big]^k
    {\sin(\frac{5\pi}{4})\sin(\frac{3\theta_T}{8})}}
    {\big[\sin(\frac{5\pi}{4}+\frac{3\theta_T}{8})\sin(\frac{5\pi}{8}-\frac{3\theta_T}{8})\big]^{k+2}}
    \end{align}
   	Moreover, the difference of the weight of $\gamma$ in $\S_T(\Theta)$ and $\S_T(\tilde \Theta)$ 
	comes only from the arcs $\chi_1,\dots, \chi_\ell$: 
	\begin{align*}
		\frac{\weight_\Theta(\gamma)}{\weight_{\tilde \Theta}(\gamma)} = \prod_{j=1}^\ell \frac{\weight_{\theta_T}(\chi_j)}{\weight_{\pi/3}(\chi_j)}.
	\end{align*}
	A direct computation shows that, for any $k \geq 0$, the weight in \eqref{eq:arc_weight} is minimised when $\theta_T = \pi/3$. 
	Thus, all terms in the right-hand side of the above equality are greater than $1$, and the conclusion is reached. 
\end{proof}

\begin{corollary}
	Let $\Theta = (\theta_1,\dots, \theta_T)$ be a finite sequence of angles with $\theta_k \in [\pi/3,2\pi/3]$ for all~$k$. 
	Then for any two points $a,b$ on the left boundary of $\S_{T}(\Theta)$ we have 
	\begin{align}\label{eq:strip_bound1}
		G_{\S_T(\Theta)} (a,b) \geq G_{\S_T(\pi/3,\theta_1,\theta_2,\dots,\theta_{T-1})} (a,b).
	\end{align}
	Additionally, 
	\begin{align}\label{eq:strip_bound2}
		G_{\S_T(\Theta)} (a,b) \geq G_{\S_T(\pi/3)} (a,b),
	\end{align}
	where the right hand side is the strip of width $T$ with all angles equal to $\pi/3$. 
\end{corollary}

\begin{proof}
	With the notation above, Lemma~\ref{lem:saw-strip-monotonicity} states that 
	$$ G_{\S_T(\Theta)} (a,b) \geq G_{\S_T(\theta_1,\theta_2,\dots,\theta_{T-1},\pi/3)} (a,b).$$
	Apply Proposition \ref{prop:saw-YB-strip} to deduce that 
	$$ G_{\S_T(\pi/3,\theta_1,\theta_2,\dots,\theta_{T-1})} (a,b) = G_{\S_T(\theta_1,\theta_2,\dots,\theta_{T-1},\pi/3)} (a,b).$$
	This proves the first bound \eqref{eq:strip_bound1}. 
	To obtain \eqref{eq:strip_bound2} it suffices to apply repeatedly \eqref{eq:strip_bound1}.
\end{proof}

Now we are ready to prove Theorem~\ref{thm-saw-2-point}.

\begin{proof}[Proof of Theorem~\ref{thm-saw-2-point}]
	Recall \eqref{eq_relation_strip}: $\cos{\tfrac{3\pi}{8}}A_{T,\Theta} = 1 - B_{T,\Theta}$ for any $T$ and sequence $\Theta$. 
	Applying the above to the constant sequence $\pi/3$ and keeping in mind Proposition \ref{prop:saw-bridges-hex}, we find
	\begin{align*}
		 A_{T,\pi/3}  \to \Big(\cos{\tfrac{3\pi}{8}}\Big)^{-1}, \qquad \text{ as $T \to \infty$}. 
	\end{align*}
	Now apply \eqref{eq:strip_bound2} to deduce that 
	\begin{align*}
		A_{T,\Theta} = \sum_{ L \in \Z} G_{\S_T(\Theta)}(0,L) \geq \sum_{ L \in \Z} G_{T,(\pi/3)}(0,L)= A_{T,(\pi/3)}. 
	\end{align*}
	Thus $\lim_{T\to \infty}A_{T,\Theta} \geq  \big(\cos{\tfrac{3\pi}{8}}\big)^{-1}$.
	However, from  \eqref{eq_relation_strip} applied to $\Theta$, 
	we find $A_{T,\Theta} \leq \big(\cos{\tfrac{3\pi}{8}}\big)^{-1}$ for all $T$.
	Thus  
	\begin{align*}
	\sum_{ L \in \Z} G_{\Theta}(0,L) = \sum_{ L \in \Z} \lim_{T \to \infty } G_{\S_T(\Theta)}(0,L)  =\lim_{T \to \infty }\sum_{ L \in \Z} G_{\S_T(\Theta)}(0,L)   &= \Big(\cos{\tfrac{3\pi}{8}}\Big)^{-1} \\
	&= \sum_{L\in \Z}G_{\pi/3}(0,L).
	\end{align*}
	Considering that 
	$$G_{\Theta}(0,L)\geq \lim_{T\to \infty}G_{T,\Theta}(0,L) \geq \lim_{T\to \infty}G_{T,(\pi/3)}(0,L) \geq 
	 G_{\pi/3}(0,L) \qquad \text{ for all $L \in \Z$}, $$
	we conclude that $G_{\Theta}(0,L)= G_{\pi/3}(0,L)$ for all $L$. 
	Finally, using the invariance $G_{\Theta}(a,b) = G_{\Theta}(0,b-a)$, we obtain the desired conclusion. 
\end{proof}


\subsection{Proof of Theorem~\ref{thm-saw-bridges} for general tilings}

\begin{proof}[{Proof of Theorem~\ref{thm-saw-bridges}}]
	By \eqref{eq_relation_strip}
	$$B_{T,\Theta} = 1 - \cos{\tfrac{3\pi}{8}}A_{T,\Theta}.$$
	We have shown in the previous proof that $A_{T,\Theta} \to \big(\cos{\tfrac{3\pi}{8}}\big)^{-1}$ as $T\to \infty$, which implies 
	$B_{T,\Theta} \to 0$. 
\end{proof}

\section{Critical surface fugacity}\label{sec:fugacity}

In this section we discuss self-avoiding walks with surface fugacities and prove Theorem~\ref{thm-fugacity} and Proposition~\ref{prop:y_c_strip}. 
We split the proof into several steps. First we introduce a slightly different notion of critical fugacity for walks in a strip, denoted $y_c^*(T,\Theta)$; this is then shown to be equal to $y_c(T,\Theta)$ defined in the introduction. 
Using the Yang--Baxter transformation, we show that the limit of $y_c^*(T,\Theta)$ as $T \to \infty$ does not depend on the sequence $\Theta$; in particular it is equal to that when $\Theta = \pi/3$, which is known to be equal to $1 + \sqrt 2$. 
Finally, it is shown that the critical fugacity of Theorem \ref{thm-fugacity} is indeed equal to $\lim_{T \to \infty }y_c^*(T,\Theta)$.

%

\subsection{Critical fugacity in the strip at $x = 1$}


When defining the critical fugacity in a strip, one may consider partition functions of walks, arcs or bridges. 
Below we show that the exact choice has little importance. 

For $\Theta = (\theta_k)_{1 \leq k \leq T}$ with $\theta_1 = \pi/3$ and all other angles in $[\pi/3, 2\pi/3]$,
recall the notation~\eqref{eq:weight-modified}
\[
    \weight_\Theta (\gamma;x,y) = \weight_\Theta (\gamma)\cdot x^{|\gamma|}\cdot y^{b(\gamma)}\,, \quad 
    \SAW_{T,\Theta} (x,y) = \sum_{\substack{\gamma \text{ starts at }0 \\ \gamma\subset \S_T(\Theta) }}{\weight_\Theta(\gamma;x,y)}\,.
\]
where~$|\gamma|$ is the length of~$\gamma$ and~$b(\gamma)$ is the number of visits of $\gamma$ to the left half of the rhombi adjacent to the left boundary of $\S_T(\Theta)$.


The partition functions of arcs and bridges are defined in a similar way and denoted by~$A_{T,\Theta} (x,y)$ and~$B_{T,\Theta} (x,y)$. 
Observe that for any self-avoiding walk $\gamma$ (that is starting and ending at any points of $\S_T(\Theta)$), its weight $\weight_{\Theta} (\gamma;x,y)$ may be defined as above. 

\begin{proposition}\label{prop:same_rad_conv}
    Let~$\Theta = \{\theta_1, \theta_2,\dots,\theta_T\}$, 
    where~$\theta_1 = \tfrac{\pi}{3}$ and~$\theta_i \in [\tfrac{\pi}{3},\tfrac{2\pi}{3}]$ for~$i>1$.
    Then the following series (with variable $y$) have the same radius of convergence:
    \[
    	A_{T,\Theta} (1, y) , \, B_{T,\Theta} ( 1, y)  , \, \SAW_{T,\Theta} (1, y)\,.  
    \]
    Write $y_c^*(T,\Theta)$ for the radius of convergence of the series above. 
\end{proposition}

\begin{proof}

The set of walks starting at $0$ includes the sets of arcs and bridges. Hence, for any $y >0$, we have:
\begin{align*}
    &\SAW_{T,\Theta} (1, y) \geq A_{T,\Theta} (1, y), \, B_{T,\Theta} (1, y)\, .
\end{align*}
Thus, the radius of convergence of $\SAW_{T,\Theta} (1, y)$ is smaller than those of $A_{T,\Theta} (1, y)$ and $B_{T,\Theta} (1, y)$.

In order to obtain opposite bounds, we use the decomposition of walks into bridges that was introduced by Hammersley and Welsh \cite{HW}. 
We prove the bound only for $B_{T,\Theta} (1, y)$, as for~$A_{T,\Theta} (1, y)$ the proof is completely analogous. For~$T=1$ the statement is obvious, so below we assume that~$T>1$.

Consider a walk~$\gamma$ in $\S_T(\Theta)$ starting at~$0$; $\gamma$ will be split into subpaths $\gamma_{-k},\dots, \gamma_\ell$ as described below. The decomposition is illustrated in Fig. \ref{fig:SAWA}.
Set the lowest (resp. highest) point of~$\gamma$ to be the non-empty rhombus with the smallest (resp. largest) second coordinate, and if several such rhombi exists, it is the leftmost (resp. rightmost) among them. 
Denote these rhombi by~$r_{\mathrm{bot}}$ and~$r_{\mathrm{top}}$ 
and let~$\gamma_0$ be the subpath of~$\gamma$ that links~$r_{\mathrm{bot}}$ and~$r_{\mathrm{top}}$ 
($\gamma_0$ includes $r_{\mathrm{bot}}$ or~$r_{\mathrm{top}}$ only if these are endpoints of $\gamma$).
Then~$\gamma \setminus \gamma_0$ is either empty, or one walk, or a union of two walks, depending on how many of the endpoints of~$\gamma$ are contained in~$\gamma_0$. If~$\gamma = \gamma_0$, the decomposition stops.
Otherwise write $\gamma^-$ for the part of $\gamma$ preceding $\gamma_0$ and $\gamma^+$ for the part following $\gamma_0$. 
We continue by decomposing $\gamma^+$ and $\gamma^-$ in the same fashion:
Suppose $\gamma^+$ is not empty and consider its lowest and the highest points. Define~$\gamma_{1}$ as the segment between these points. 
Note that now $\gamma^+ \setminus \gamma_1$ is formed of at most one walk, not two. 
Continue decomposing $\gamma^+ \setminus \gamma_1$ to obtain $\gamma_2$ etc, until the remaining walk is empty. 
Apply the same procedure to decompose $\gamma^-$ into $\gamma_{-1}, \gamma_{-2},$ etc.

\begin{figure}
    \centering
      \includegraphics[scale=0.64]{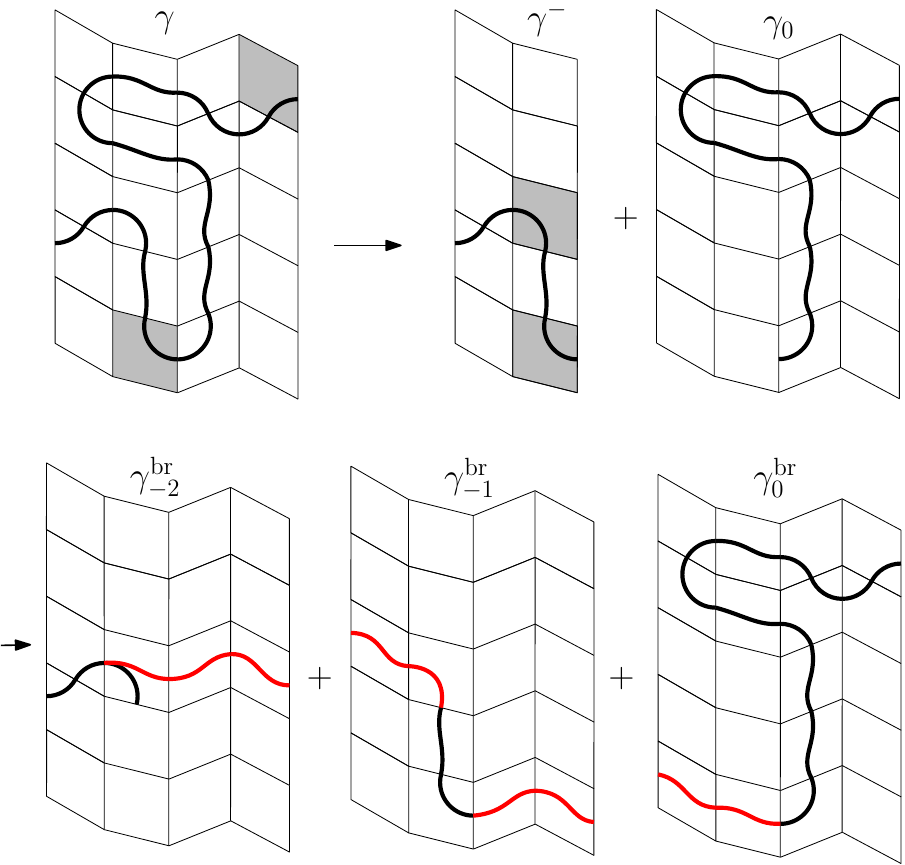}
    \caption{{\em Top Left:} A walk $\gamma$ in $\S_T(\Theta)$ with $r_{\mathrm{bot}}$ and~$r_{\mathrm{top}}$ marked in gray.
    {\em Top Right:} The decomposition of $\gamma$ in $\gamma^-$ and $\gamma_0$; $\gamma^+$ is void.
       {\em Bottom:} The further decomposition of $\gamma$ into basic pieces. These are completed by the red paths to form bridges.}
    \label{fig:SAWA}
\end{figure}

Importantly, in this way~$\gamma$ gets split in at most~$2T-1$ pieces. 
Indeed, the left-most points of $\gamma_{0},\gamma_1,\dots, \gamma_\ell$ are each strictly to the right of the preceding one. Thus $\ell < T$. 
Similarly, the right-most points of $\gamma_{0},\gamma_{-1},\dots, \gamma_{-k}$ are each strictly to the left of the preceding one, and $k <T.$
%

In general, it is not true that the weight of~$\gamma$ is equal to the product of the weights of the pieces obtained above, because
the rhombi containing~2 arcs in different pieces contribute~$w_1$ (or~$w_2$) to the weight of~$\gamma$ and~$u_1^2$ (or~$u_2^2$) to the product of the weights of the pieces. However, since~$u_1(\theta)^2 \geq w_1(\theta)$ and~$u_2(\theta)^2 \geq w_2(\theta)$ for any~$\theta \in [\tfrac{\pi}{3}, \tfrac{2\pi}{3}]$, we obtain the following inequality:
\begin{equation}
\label{eq_compare_weights_pieces_walk}
	\weight (\gamma ; 1, y) \leq \prod_{i=-k}^{\ell} {\weight(\gamma_i ; 1, y)}.
\end{equation}
Now complement the walks~$\gamma_i$ to create bridges by adding straight lines in the rhombi lying to the left (resp. right) of the lower (resp. upper) endpoint of~$\gamma_i$ and contained in the same rows as the endpoints (see Fig. \ref{fig:SAWA}). Small local modifications may be needed to glue the added paths to $\gamma_i$. Denote the resulting bridges by~$\gamma_i^{\mathrm{br}}$. Note that by the choice of~$\gamma_i$, the walks~$\gamma_i^{\mathrm{br}}$ do not have self-intersections. The walks~$\gamma_i$ and~$\gamma_i^{\mathrm{br}}$ differ by at most~$2T$ rhombi, which are empty for $\gamma_i$ but contain straight lines for $\gamma_i^{\mathrm{br}}$. Thus
\[
\weight(\gamma_i ; 1, y) \leq \frac{1}{v(\Theta)y} \weight(\gamma_i^{\mathrm{br}} ; 1, y)\, ,
\]
where~$v(\Theta)> 0$ is some constant which depends on $T$ and $\Theta$ only. 
Recall that there are at most~$2T-1$ pieces~$\gamma_i$. From this, the previous inequality and~\eqref{eq_compare_weights_pieces_walk}, we obtain:
\[
	\weight (\gamma; 1, y) \leq\frac{1}{[v(\Theta)y]^{2T-1}}  \prod_i {\weight(\gamma_i^{\mathrm{br}}; 1, y)}\, .
\]
Sum this inequality over all possible choices of~$\gamma$.  Using again that there are at most~$2T-1$ walks in the decomposition, 
the right-hand side can be bounded by the partition function of bridges:
\begin{equation}
\label{eq:compare_part_func_arc_walk}
    \SAW_{T,\Theta}(1, y) 
    \leq \frac{1}{[v(\Theta)y]^{2T-1}}\sum_{\gamma \, : \, 0 \to z, \, \gamma \subset \Omega_T}
    \prod_i {\weight(\gamma_i^{\mathrm{br}};1,y)} \leq \left[\frac{4T}{v(\Theta)y}\right]^{2T-1} (1+B_{T,\Theta}(1, y))^{2T-1},
\end{equation}
where the additional factor~$4T$ in the right hand side is due to the reconstruction cost of~$\gamma$ given~$(\gamma_i^{\mathrm{br}})_{i\in [-k,\ell]}$.

Hence, the radius of convergence of~$B_{T,\Theta}(1, y)$ and~$\SAW_{T,\Theta}(1, y)$ is the same. 

The same strategy may be used to show that $A_{T,\Theta}(1, y)$ and~$\SAW_{T,\Theta}(1, y)$ have the same radius of convergence. The only difference is that this time the subpaths $\gamma_i$ should be transformed into arcs rather than bridges.
\end{proof}

\subsection{Critical fugacity in the strip: $y_c^*(T,\Theta) = y_c(T,\Theta)$. }

Recall that the critical fugacity in a strip was defined in the introduction as 
\[
y_c(T,\Theta) = \sup \{y\, | \, \forall 0<x<1, \ \SAW_{T,\Theta} (x,y) < \infty \}.
\]
We show now that the two notions of critical fugacity in a strip, namely $y_c(T,\Theta)$ and $y_c^*(T,\Theta)$, coincide. 

\begin{proposition}\label{prop:fug_2_def_strip}
    Let~$\Theta= \{\theta_k\}_{k=1}^T$, where~$\theta_1= \tfrac{\pi}{3}$ and~$\theta_k\in [\tfrac{\pi}{3},\tfrac{2\pi}{3}]$ for~$k>1$. 
    Then~$y_c(T,\Theta) = y_c^*(T,\Theta)$.
\end{proposition}

We start by a technical lemma which in effect states that a walk in a strip has a positive density of points on the boundary. 
Such a result is in the spirit of Kesten's pattern theorem \cite{kestenone}.
For completeness and simplicity, we provide a proof with no reference to Kesten's result. 

\begin{lemma}\label{lem:fug_xy}
    Let~$\Theta= \{\theta_k\}_{k=1}^T$, where~$\theta_1= \tfrac{\pi}{3}$ and~$\theta_k\in [\tfrac{\pi}{3},\tfrac{2\pi}{3}]$ for~$k>1$. 
    Then there exists a constant~$C(T)>0$ which depends only on~$T$, such that for any~$0<x\leq1$ and~$y>1$
\begin{align}
 \SAW_{T,\Theta}(x,y) &\leq \SAW_{T,\Theta}(xy,1), \label{ineq:fug_xy}\\
 \SAW_{T,\Theta}(x,x^{-C}y) &\geq \SAW_{T,\Theta}(1,y)\,.       \label{ineq:fug_x_c_y}
\end{align}
\end{lemma}

\begin{proof}
Inequality~\eqref{ineq:fug_xy} follows from the fact that the length of a walk is greater than the number of times it visits the boundary.

Inequality~\eqref{ineq:fug_x_c_y} is proven by altering arbitrary walks $\gamma$ to form walks $\gamma^{\text{fug}}$ which have a positive density of points on the left boundary. We describe the map $\gamma \mapsto \gamma^{\text{fug}}$ next.

Recall the indexing of the rows of $\S_T(\Theta)$ by $\mathbb Z$. 
Call  \emph{a marked line} of $\S_T(\Theta)$ the collection of edges separating rows~$(k+\tfrac12)T$ and~$(k+\tfrac12)T+1$ with~$k\in\mathbb{Z}$. 
Let $\gamma$ be a walk on $\S_T(\Theta)$ starting at $0$.
To define~$\gamma^{\text{fug}}$ insert at each marked line two rows of rhombi, containing arcs as described below.
Fix a marked line~$\ell$, the two rows of rhombi inserted at $\ell$ contain:
\begin{itemize*}
\item for each point in~$\gamma\cap \ell$ except the leftmost one, insert two straight vertical arcs of type $v$;
\item for the leftmost point in~$\gamma\cap \ell$, insert a path contained in the two inserted rows that, when viewed from bottom to top, travels left in the lower row, touches the first column turning upwards, then travels back right using the upper row (if the left-most point is in the first column, complete the added rhombi as in the point above);
\item all rhombi not affected by this procedure are void. 
\end{itemize*}
Perform this for all marked lines. Note that when marked lines are not crossed by $\gamma$, the added rows only contain empty rhombi. 
It is easy to see that the result of this procedure is a self-avoiding walk on $\S_T(\Theta)$, which we call $\gamma^{\text{fug}}$.
See Fig. \ref{fig:pattern} for an example. 

\begin{figure}
    \centering
      \includegraphics[scale=0.7]{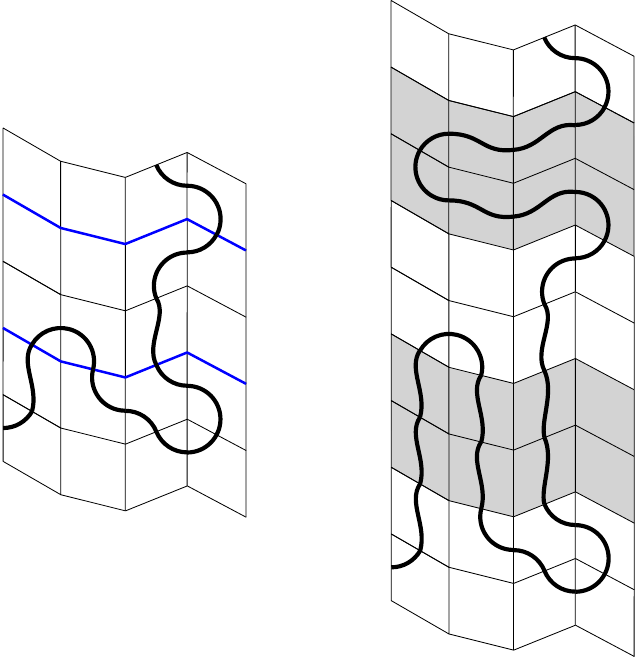}
    \caption{{\em Left:} A walk $\gamma$ in $\S_T(\Theta)$ crossing two marked lines (blue). {\em Right:} The associated walk $\gamma^{\text{fug}}$; the added rows are marked in gray.}
    \label{fig:pattern}
\end{figure}

The map~$\gamma \mapsto \gamma^{\text{fug}}$ is injective. Indeed it suffices to delete the added rows (whose indices are deterministic) to retrieve $\gamma$ from $\gamma^{\text{fug}}$. Thus
\begin{align}\label{eq:inj}
\SAW_{T,\Theta}(x,x^{-C}y) \geq \sum_{\gamma} \weight_{\Theta}(\gamma^{\text{fug}};x,x^{-C}y), \qquad \text{for all $C>0$,}
\end{align}
since in the right hand side we only sum the weight of images of walks by the map defined above. 
 
Now observe that, since the length of $\gamma$ inside any rhombus is at most~4,~$\gamma$ crosses at least $|\gamma|/(4T^2)$ marked lines. 
Each marked line generates at least one contribution to the fugacity for $\gamma^{\text{fug}}$, thus $b(\gamma^{\text{fug}}) \geq |\gamma|/(4T^2) $. 
On the other hand, $\gamma$ visits at most $2|\gamma|/T$ marked lines and for each such line the added rhombi contain a total length of arcs of at most $8T$. 
Thus $|\gamma^{\text{fug}}| - |\gamma| \leq 16|\gamma|.$
In conclusion 
\begin{align*}
	\frac{b(\gamma^{\text{fug}})}{|\gamma^{\text{fug}}| - |\gamma|} \geq \frac{1}{64 T^2} =:\frac1C.
\end{align*}
In particular 
\begin{align*}
	\frac{\weight_{\Theta}(\gamma^{\text{fug}};x,x^{-C}y)}{\weight_{\Theta}(\gamma;x,y)}
	= x^{|\gamma^{\text{fug}}| -  |\gamma| - Cb(\gamma^{\text{fug}}) }y^{b(\gamma^{\text{fug}}) - b(\gamma)}
	\geq 1,
\end{align*}
since the exponents for $x$ and $y$ are negative and positive, respectively. 
Inserting this into \eqref{eq:inj} we find
\begin{align*}
	\SAW_{T,\Theta}(x,x^{-C}y)
	&\geq  \sum_{\gamma} \weight_{\Theta}(\gamma^{\text{fug}};x,x^{-C}y)\\
	&\geq \sum_{\gamma} \weight_{\Theta}(\gamma;1,y)
	=\SAW_{T,\Theta}(1,y).
	\qedhere
\end{align*}

\end{proof}

\begin{proof}[Proof of Proposition~\ref{prop:fug_2_def_strip}]
First we show the inequality $y_c(T,\Theta) \ge y_c^*(T,\Theta)$. 
Take~$y>y_c(T,\Theta)$. Then for~$x<1$ large enough,~$\SAW_{T,\Theta}(x,y)$ diverges. 
By Ineq.~\eqref{ineq:fug_xy}, one has that $\SAW_{T,\Theta}(xy;1)$ diverges as well. Hence~$xy\ge y_c^*(T,\Theta)$. 
Since~$x$ may be arbitrarily close to~1, we proved that~$y \ge y_c^*(T,\Theta)$.
By choice of $y$ this implies  $y_c(T,\Theta) \ge y_c^*(T,\Theta)$.

Let us now show the converse inequality~$y_c^*(T,\Theta) \ge y_c(T,\Theta)$. 
Take~$y>y_c^*(T,\Theta)$. Then $\SAW_{T,\Theta}(1;y)$ diverges. Use Ineq.~\eqref{ineq:fug_x_c_y} to see that~$\SAW_{T,\Theta}(x,x^{-C}y)$ diverges as well for any $x < 1$, where $C = C(T) > 0$ is given by Lemma~\ref{lem:fug_xy}. Thus~$x^{-C}y \geq y_c(T,\Theta)$ for all~$x < 1$, which implies that~$y \ge y_c(T,\Theta)$.
Since $y>y_c^*(T,\Theta)$ is arbitrary, we proved $y_c^*(T,\Theta) \ge y_c(T,\Theta)$.
\end{proof}

\subsection{Critical fugacities in strips do not depend on $\Theta$}

Our next goal is to show that~$y_c(T,\Theta) \to 1+\sqrt{2}$, i.e. that the critical fugacities on strips of rhombi converge to the critical fugacity on the hexagonal lattice, which corresponds to the case when all  rhombi have angle~$\pi/3$. 

By Proposition~\ref{prop:fug_2_def_strip}, $y_c(T,\Theta)$ is the radius of convergence of~$\SAW_{T,\Theta} (1,y)$. In the spirit of notation we introduced before, we denote by~$y_c(T,\pi/3)$ the radius of convergence of the series~$\SAW_{T,\pi/3} (1,y)$, i.e. in the case when all rhombi have angle~$\tfrac{\pi}{3}$.  In the next lemma, it is shown that~$y_c(T,\Theta)$ can only increase, when the rightmost column of rhombi is erased, or when all angles of the rhombi are changed to~$\tfrac{\pi}{3}$.

\begin{lemma}
    \label{lem:fug_monotone}
    Let~$\Theta= (\theta_k)_{k\geq1}$ be such that~$\theta_1= \tfrac{\pi}{3}$ and~$\theta_k\in [\tfrac{\pi}{3},\tfrac{2\pi}{3}]$ for~$k>1$and $T \geq 2$. Then
    \begin{itemize*}
    \item[(i)] $y_c(T,\pi/3) \geq y_c(T,\Theta) $;
    \item[(ii)] $y_c(T-1,\Theta)\geq y_c(T,\Theta)$.
    \end{itemize*}
\end{lemma}

\begin{proof}
(i)  By Proposition~\ref{prop:same_rad_conv}, it is enough to show that for any~$y \geq 0$ one has~$A_{T,\Theta} (1,y)\ge A_{\pi/3, T} (1,y)$. This inequality was shown in Lemma~\ref{lem:saw-strip-monotonicity} in the absence of surface fugacities. It is easy to check that the proof adapts straightforwardly when fugacities are added on the left side. Indeed the proof is based on Yang-Baxter transformations that do not affect the left-most column, since this one already has angle $\pi/3$. 

(ii) The inequality~$A_{T,\Theta} (y)\ge A_{T-1,\Theta} (y)$ is trivial, since all walks contributing to the right hand side also contribute to the left hand side. The inequality on the radii of convergence follows readily. 
\end{proof}

Now we are ready to finish the proof of Proposition~\ref{prop:y_c_strip} by showing that~$y_c(T,\Theta) \to 1+\sqrt{2}$.


\begin{proof}[Proof of Proposition~\ref{prop:y_c_strip}]
In~\cite{BBDDG} it was shown that the critical surface fugacity on the hexagonal lattice is equal to~$1+\sqrt{2}$. In particular, Corollary~8 in~\cite{BBDDG} implies that~$y_c^*(\pi/3, T) \to 1+\sqrt{2}$. In Lemma~\ref{lem:fug_monotone} it is shown that~$y_c^*(\pi/3, T) \geq y_c^*(T,\Theta) $, for any~$T$. Hence, 
$$ \lim_{T \to \infty} y_c^*(T,\Theta) \leq 1+\sqrt{2}.$$
The existence of the limit above is ensured by the monotonicity of $y_c^*(T,\Theta)$ in~$T$.

The opposite inequality follows directly from Corollary~\ref{cor:small_y}. Indeed, suppose that~$\lim y_c^*(T,\Theta)<1+\sqrt{2}$. Then for some~$T$, one has~$y_c^*(T,\Theta)<1+\sqrt{2}$. Consider a value of~$y$ between~$y_c^*(T,\Theta)$ and~$1+\sqrt{2}$ and note that by Corollary~\ref{cor:small_y}, $B_{T,\Theta} (1,y) = B_{T} (\Theta)(y) \leq \tfrac{\sqrt2 y}{1+\sqrt{2}-y}$.
This contradicts the assumption that~$y>y_c^*(T,\Theta)$, that is the radius of convergence of~$B_{T,\Theta}(1,\cdot)$.
\end{proof}

\subsection{Critical fugacity in half-plane: proof of Theorem \ref{thm-fugacity}}

In order to prove Theorem~\ref{thm-fugacity}, it remains to show that~$y_c = 1+\sqrt{2}$. 
Recall that~$y_c$ is defined as the supremum of all~$y$ such that~$\SAW_\Theta (x,y)$ is finite for all~$x<1$.

\begin{proof}[Proof of Theorem~\ref{thm-fugacity}]
    We will proceed by double inequality. 
    Let~$y> 1+\sqrt{2}$. 
    By Proposition~\ref{prop:y_c_strip}, there exists ~$T$ such that~$y> y_c(T,\Theta)$. Hence, by the definition of~$y_c(T,\Theta)$, there exists~$0<x<1$ such that~$\SAW_{T,\Theta} (x,y) = \infty$. 
    Since $\SAW_{T,\Theta} (x,y) \leq \SAW_\Theta (x,y)$, the latter diverges as well. This implies that~$y\ge y_c$. 
    Recall that~$y$ was chosen arbitrarily greater than~$1+\sqrt{2}$, thus, $y_c \leq 1+\sqrt{2}$.
    
	The opposite inequality is based on the results obtained through the parafermionic observable with fugacity.
    Take~$1 \leq y<1+\sqrt{2}$. By Corollary~\ref{cor:small_y},~$B_{T,\Theta}(1,y)<c$, where~$c$ is a constant that depends only on~$y$. 
    Note that all walks which contribute to~$B_{T,\Theta}(1,y)$ have to cross at least~$T$ rhombi. 
    Thus,~$B_{T,\Theta}(x,y) < x^T\cdot c$, 
    and $\sum_{T \geq 1}B_{T,\Theta}(x,y) < \frac{c}{1-x}< \infty$ for all $x < 1$. 
   
   	Fix $x < 1$. 
   	Let us now prove that $\SAW_\Theta(x,y) < \infty$. Write $\Theta'$ for the sequence $(\theta_2,\theta_3,\dots)$.
	Let $\gamma$ be a walk in $\H(\Theta)$. 
	Write $\gamma$ as the concatenation of two walks $\gamma^{(a)}$ and $\gamma^{(w)}$, 
	where $\gamma^{(a)}$ ends at the last visit of $\gamma$ of column $1$. 
	The walk $\gamma^{(w)}$ is contained in columns $2,3,\dots$ and hence does not feel the effect of the fugacity. 
	Thus it may be viewed as a walk in $\H(\Theta')$ with weight $\weight_{\Theta'}(\gamma^{(w)};x,1)$.
	
	Further split $\gamma^{(a)}$ in two walks: 
	$\gamma^{(1)}$ is the walk from the starting point to the first point of $\gamma^{(a)}$ in the right-most column visited by $\gamma^{(a)}$ (write $T$ for the index of this column); 
	$\gamma^{(2)}$ is simply $\gamma^{(a)}\setminus \gamma^{(1)}$. 
	The endpoints of $\gamma^{(1)}$ and $\gamma^{(2)}$ may be modified locally to create two bridges $\gamma^{(b1)}$ and $\gamma^{(b2)}$ in $\S_T(\Theta)$.
	Due to the local modifications, there exists a universal constant $\delta>0$ such that 
	\begin{align*}
		\weight_\Theta(\gamma^{(a)};x,y) \le  \weight_\Theta(\gamma^{(1)};x,y)\weight_\Theta(\gamma^{(2)};x,y) 
		\leq \delta \weight_\Theta(\gamma^{(b1)};x,y)\weight_\Theta(\gamma^{(b2)};x,y).
	\end{align*}
	Thus we associated to $\gamma$ a triplet $\gamma^{(b1)}, \gamma^{(b2)}, \gamma^{(w)}$, the first two being bridges in a certain $\S_T(\Theta)$ and the third being a walk in $\H_T(\Theta')$. 
	This operation is clearly injective, and we find
	\begin{align*}
		\SAW_\Theta(x,y) 
		&\leq \sum_{\gamma}\weight_\Theta(\gamma^{(1)};x,y)\weight_\Theta(\gamma^{(2)};x,y)\weight_\Theta(\gamma^{(w)};x,y)\leq \delta \sum_{T \geq 1} B_{T, \Theta}(x,y)^2 \ \SAW_{\Theta'}(x,1)\\
		&\leq \delta \Big[\sum_{T \geq 1} B_{T, \Theta}(x,y)\Big]^2 \SAW_{\Theta'}(x,1)\leq \delta \Big(\frac{c}{1-x}\Big)^2\SAW_{\Theta'}(x,1).
	\end{align*}
	Finally, since $x<1$, $\SAW_{\Theta'}(x,1)< \infty$ which implies $\SAW_\Theta(x,y) < \infty$. 
	Since $x < 1$ is arbitrary, this shows that $y < y_c$, and thus that $y_c \geq 1 + \sqrt 2$. 
\end{proof}

\bibliographystyle{alpha}

\bibliography{SAW}

\Addresses

\end{document}